\documentclass[a4paper,10pt]{article}

\usepackage[english]{babel}
\usepackage{a4wide}
\usepackage[utf8]{inputenc}
\usepackage{amssymb,amsmath,amscd,amsfonts,amsthm,bbm,mathrsfs,enumerate}
\usepackage[shortlabels]{enumitem}
\usepackage[colorlinks=true, linkcolor=blue, citecolor=blue]{hyperref}
\usepackage{graphicx} 
\usepackage{tikz,pgfplots}
\usepackage{color} 
\usepackage{yhmath}
\usepackage{accents} 

\usepackage[textwidth=1.8cm, textsize=footnotesize]{todonotes}

\DeclareUnicodeCharacter{00A0}{~}

\theoremstyle{definition}
\newtheorem{theorem}{Theorem}[section]
\newtheorem*{theorem*}{Statement}
\newtheorem{lemma}[theorem]{Lemma}
\newtheorem{corollary}[theorem]{Corollary}
\newtheorem{proposition}[theorem]{Proposition}

\newtheorem*{condition*}{Condition}
\newtheorem{assumption}{Assumption}[section]
\newtheorem*{assumption*}{Assumption}

\DeclareMathOperator{\dom}{dom}

\DeclareMathOperator{\inter}{int}
\DeclareMathOperator{\relint}{ri}
\DeclareMathOperator{\epi}{epi}
\DeclareMathOperator{\graph}{gr}
\DeclareMathOperator{\prox}{prox}

\DeclareMathOperator{\support}{supp}

\DeclareMathOperator{\lev}{lev}

\DeclareMathOperator{\cl}{cl}

\DeclareMathOperator{\distC}{\mathsf d}

\newcommand{\1}{\mathbbm 1}

\newcommand{\sx}{{\mathsf x}}
\newcommand{\sy}{{\mathsf y}}
\newcommand{\sz}{{\mathsf z}}
\newcommand{\sw}{{\mathsf w}}

\newcommand{\sH}{{\mathsf H}}

\newcommand{\bP}{{{\mathbb P}}} 
\newcommand{\bE}{{{\mathbb E}}} 
\newcommand{\bN}{{{\mathbb N}}} 

\newcommand{\mA}{{\mathcal A}} 
\newcommand{\mB}{{\mathcal B}} 
\newcommand{\mC}{{\mathcal C}} 
\newcommand{\mD}{{\mathcal D}} 
\newcommand{\mU}{{\mathcal U}}

\newcommand{\bmD}{\cl({\mathcal D})} 

\newcommand{\sA}{{\mathsf A}}
\newcommand{\sB}{{\mathsf B}}
\newcommand{\sJ}{{\mathsf J}}
\newcommand{\sX}{{\mathsf X}}
\newcommand{\sG}{{\mathsf G}}

\newcommand{\maxmon}{{\mathscr M}} 
\newcommand{\Selec}{{\mathfrak S}} 

\newcommand{\mcA}{{\mathscr A}} 
\newcommand{\mcB}{{\mathscr B}} 
\newcommand{\mcN}{{\mathscr N}}

\newcommand{\mcF}{{\mathscr F}} 
\newcommand{\mcG}{{\mathscr G}}

\newcommand{\cP}{{{\mathcal P}}} 
\newcommand{\cS}{{{\mathcal S}}}

\newcommand{\cM}{{{\mathcal M}}} 
\newcommand{\cD}{{{\mathcal D}}}

\newcommand{\cL}{{{\mathcal L}}}

\newcommand{\cK}{{{\mathcal K}}} 
\newcommand{\cI}{{{\mathcal I}}} 

\newcommand{\bR}{{{\mathbb R}}}

\newcommand{\ps}[1]{\langle #1 \rangle}

\newcommand{\bs}{\boldsymbol}
\newcommand{\eqdef}{:=} 
\begin{document}

\title{A constant step Forward-Backward algorithm involving random
maximal monotone operators}



\author{Pascal Bianchi$^*$ \and Walid Hachem$^\dagger$ \and 
Adil Salim\thanks{LTCI, T\'el\'ecom ParisTech, Universit\'e 
 Paris-Saclay. \newline
46, rue Barrault, 75634 Paris Cedex 13, France.
\newline
$\dagger$CNRS / LIGM (UMR 8049), Universit\'e Paris-Est Marne-la-Vallée. 
\newline 
 5, boulevard Descartes, Champs-sur-Marne, 77454, Marne-la-Vallée Cedex 2, 
France.
\newline
This work was supported by the Agence Nationale pour la Recherche,
France, (ODISSEE project, ANR-13-ASTR-0030) and by the Labex Digiteo-DigiCosme
(OPALE project), Universit\'e Paris-Saclay.
\newline
\texttt{pascal.bianchi@telecom-paristech.fr, walid.hachem@u-pem.fr, 
adil.salim@telecom-paristech.fr}}
} 

\date{\today}

\maketitle

\begin{abstract} 
A stochastic Forward-Backward algorithm with a constant step is studied.  At
each time step, this algorithm involves an independent copy of a couple of
random maximal monotone operators.  Defining a mean operator as a selection
integral, the differential inclusion built from the sum of the two mean
operators is considered.  As a first result, it is shown that the interpolated
process obtained from the iterates converges narrowly in the small step regime
to the solution of this differential inclusion.  In order to control the long
term behavior of the iterates, a stability result is needed in addition.  To
this end, the sequence of the iterates is seen as a homogeneous Feller Markov
chain whose transition kernel is parameterized by the algorithm step size. 
The cluster points of the Markov chains invariant measures in the small step
regime are invariant for the semiflow induced by the differential inclusion.
Conclusions regarding the long run behavior of the iterates for small steps are
drawn.  It is shown that when the sum of the mean operators is demipositive,
the probabilities that the iterates are away from the set of zeros of this sum
are small in Ces\`aro mean.  The ergodic behavior of these iterates is studied
as well. Applications of the proposed algorithm are considered.  In particular,
a detailed analysis of the random proximal gradient algorithm with constant
step is performed.  
\end{abstract}

{\bf Keywords: }  
  Dynamical systems,  
  Narrow convergence of stochastic processes,    
  Random maximal monotone operators,  
  Stochastic approximation with constant step, 
  Stochastic Forward - Backward algorithm, 
  Stochastic proximal point algorithm. \\
{47H05, 47N10, 62L20, 34A60.}

\section{Introduction} 
\label{sec-intro}

Given two maximal monotone operators $\sA$ and $\sB$ on the space $E=\bR^N$, 
where $\sB$ is single valued, the Forward-Backward splitting algorithm is an
iterative algorithm for finding a zero of the sum operator $\sA + \sB$. It
reads 
\begin{equation} 
\label{fb-deterministe} 
x_{n+1} = ( I + \gamma \sA )^{-1} ( x_n - \gamma \sB(x_n) ) \, , 
\end{equation} 
where $\gamma$ is a positive step. This algorithm consists in a forward 
step $(I - \gamma \sB) (x_n)$ followed by a backward step, where 
the resolvent $( I + \gamma \sA )^{-1}$ of $\sA$, known to be single valued
as $\sA$ is maximal monotone, is applied to the output of
the former. When $\sB$ satisfies a so called cocoercivity condition, and when
the step $\gamma$ is small enough, the convergence of the algorithm towards a 
zero of $\sA + \sB$ (provided it exists) is a well established fact
\cite[Ch.~25]{bau-com-livre11}. 
In the field of convex optimization, this algorithm can be used to find a
minimizer of the sum of two real functions $F + G$ on $E$, where $F$ is a
convex function which is defined on the whole $E$ and which has a Lipschitz
gradient, and where $G$ is a convex, proper, and lower semi continuous (lsc)
function. In this case, the Forward-Backward algorithm is known as the proximal
gradient algorithm, and is written as 
$x_{n+1} = \prox_{\gamma G}(x_n-\gamma\nabla F(x_n))$, where 
$\prox_{\gamma G} \eqdef (I + \gamma\partial G)^{-1}$ is Moreau's proximity
operator of $\gamma G$. 

In this paper, we are interested in the situation where the operators 
$\sA$ and $\sB$ are replaced with random maximal monotone operators. 
Denote as $\maxmon$ the set of maximal monotone operators on $E$, let 
$A , B : \Xi \to \maxmon$ be two functions from a measurable space 
$(\Xi, \mcG)$ to $\maxmon$, and let $(\xi_n)$ be a sequence of independent and
identically distributed (iid) random variables from some probability space
to $(\Xi, \mcG)$ with the probability distribution $\mu$. Assuming that 
$B(s)$ is single-valued operator defined on the whole $E$, we examine the 
stochastic version of the Forward-Backward algorithm 
\begin{equation}
\label{fb-sto-intro} 
x_{n+1} = ( I + \gamma A(\xi_{n+1}))^{-1} ( I - \gamma B(\xi_{n+1})) x_n \, ,  
\quad \gamma > 0 \, . 
\end{equation} 
Our aim is to study the dynamical behavior of this algorithm in the limit of 
the small steps $\gamma$, where the effect of the noise due to the $\xi_n$ 
will be smoothened. 

To give an application example for this algorithm, let us consider again the
minimization problem of the sum $F+G$, and let us assume that these functions
are unknown to the observer (or difficult to compute), and are written as 
$F(x) = \bE_{\xi_1} f(\xi_1, x)$ and $G(x) = \bE_{\xi_1} g(\xi_1, x)$.  When
the functions $f$ and $g$ are known with $f(\xi_1,\cdot)$ being convex
differentiable, and $g(\xi_1,\cdot)$ being convex, proper, and lsc, and when an 
iid sequence $(\xi_n)$ is available, we can approximatively solve 
the minimization problem of $F+G$ by resorting to the stochastic proximal 
gradient algorithm
$x_{n+1} = \prox_{\gamma g(\xi_{n+1}, \cdot)} 
   (x_n - \gamma\nabla_x f(\xi_{n+1}, x_n))$. 
Similar algorithms has been studied in \cite{bia-hac-16,ros-vil-vu-16} with the additional
assumption that the step size $\gamma$ vanishes as $n$ tends to infinity.
The main asset of such vanishing step size algorithms 
is that the iterates (with or without averaging) converge almost surely as the iteration index goes to infinity.
This paper focuses on the case where the step size $\gamma$ is fixed w.r.t. $n$.
As we shall see below, convergence hold in a weaker sense in this case.
Loosely speaking, the iterates fluctuate in a small neighborhood of the set of sought solutions, but do not converge in an almost sure sense
as $n\to \infty$.
Yet, constant step size algorithms have raised a great deal of attention in the signal processing and machine learning literature (\cite{dieuleveut2017bridging}). First, they are known to reach a neighborhood of the solution in a fewer number of iterations than the decreasing step algorithms. Second, they are in practice able to adapt to non stationary or slowly changing environments,
and thus track a possible changing set of solutions. This is particularly helpful in adaptive signal processing
for instance.\\

In order to study the dynamical behavior of \eqref{fb-sto-intro}, we introduce 
the operators 
\[
\mA = \int A(s) \, \mu(ds) \quad \text{and} \quad 
\mB = \int B(s) \, \mu(ds)  \, , 
\]
where the first integral is a set-valued integral, which is to be recognized as
a \emph{selection integral} \cite{molchanov2006theory}. Assuming that the 
monotone operator $\mA + \mB$ is maximal, it is a standard fact of the 
monotone operator theory that for any $x_0$ in the domain of $\mA + \mB$, the 
Differential Inclusion (DI) 
\begin{equation} 
\label{eq:di-intro}
\left\{\begin{array}[h]{lcl}
\dot \sx(t) &\in& - (\mA + \mB)(\sx(t)) \,  \\
 \sx(0) &=& x_0
\end{array}\right. 
\end{equation} 
admits a unique absolutely continuous solution on $\bR_+ \eqdef[0,\infty)$
\cite{bre-livre73,aub-cel-(livre)84}. 
Let $\sx_\gamma(t)$ be the continuous random process obtained by assuming
that the iterates $x_n$ are distant apart by the time step $\gamma$, and by
interpolating linearly these iterates. Then, the first step of the approach
undertaken in this paper is to show that $\sx_\gamma$ shadows the solution of
the DI for small $\gamma$, in the sense that it converges narrowly to this
solution as $\gamma\to 0$ in the topology of convergence on the compact sets of
$\bR_+$. The same idea is behind the so-called ODE method which is frequently
used in the stochastic approximation literature 
\cite{ben-(cours)99, kus-yin-(livre)03}. 

The compact convergence alone is not enough to control the long term behavior
of the iterates. A stability result is needed. To that end, the second step of
the approach is to view the sequence $(x_n)$ as a homogeneous Feller Markov
chain whose transition kernel is parameterized by $\gamma$. In this context,
the aim is to show that the set of invariant measures for this kernel is non
empty, and that the family of invariant measures obtained for all $\gamma$
belonging to some interval $(0,\gamma_0]$ is tight.  We shall obtain a general
tightness criterion which will be made more explicit in a number of situations
of interest involving random maximal monotone operators.

The narrow convergence of $\sx_\gamma$, together with the tightness of the
Markov chain invariant measures, lead to the invariance of the small $\gamma$
cluster points of these invariant measures with respect to the semiflow induced
by the DI \eqref{eq:di-intro} (see \cite{has-63,for-pag-99,ben-hir-aap99} for
similar contexts). Using these results, it becomes possible to characterize the
long run behavior of the iterates $(x_n)$. In particular, the proximity of
these iterates to the set of zeros $Z(\mA+\mB)$ of $\mA+\mB$ is of obvious
interest. First, we show that when the operator $\mA+\mB$ is
\emph{demipositive} \cite{bru-75}, the probabilities that the iterates are away
from $Z(\mA+\mB)$ are small in Ces\`aro mean.  Whether $\mA+\mB$ is
demipositive or not, we can also characterize the ergodic behavior of the
algorithm, showing that when $\gamma$ is small, the partial sums $n^{-1}
\sum_1^n x_k$ stay close to $Z(\mA+\mB)$ with a high probability.

Stochastic approximations with differential inclusions were considered in
\cite{ben-hof-sor-05} and in \cite{fau-rot-10} from the dynamical systems
viewpoint. The case where the DI is defined by a maximal monotone operator was
studied in \cite{bia-16}, \cite{bia-hac-16}, and \cite{ros-vil-vu-16}.
Instances of the random proximal gradient algorithm were treated in
\emph{e.g.}, \cite{atc-for-mou-14} or \cite{rosasco2014convergence}.  All these
references dealt with the decreasing step case, which requires quite different
tools from the constant step case. This case is considered in 
\cite{com-pes-pafa16} (see also \cite{com-pes-siam15}), which relies on a 
Robbins-Siegmund like approach requiring summability assumptions on the random 
errors. The constant step case is also dealt with in 
\cite{rot-san-siam13} and in \cite{bia-hac-sal-(arxiv)16} for generic
differential inclusions.  In the present work, we follow the line of reasoning
of our paper \cite{bia-hac-sal-(arxiv)16}, noting that the case where the DI is
defined by a maximal monotone operator has many specificities.  For
instance, a maximal monotone operator is not upper semi continuous in general,
as it was assumed for the differential inclusions studied in
\cite{rot-san-siam13} and \cite{bia-hac-sal-(arxiv)16}. Another difference lies
in the fact that we consider here the case where the domains of the operators
$A(s)$ can be different. Finally, the tightness criterion for the 
Markov chain invariant measures requires a quite specific treatment in the 
context of the maximal monotone operators. 

We close this paragraph by mentioning \cite{ber-11}, where one of the studied
stochastic proximal gradient algorithms can be cast in the general framework of
\eqref{fb-sto-intro}. 

\paragraph*{Paper organization.} 
Section \ref{sec-bckgrd} introduces the main algorithm and recalls some known
facts about random monotone operators and their selection integrals.
Section~\ref{sec-results} provides our assumptions and states our main result
about the long run behavior of the iterates. A brief sketch of the proof is
also provided for convenience, the detailed arguments being postponed to the
end of the paper.  Section~\ref{sec-tightness} provides some illustrations of
our results in particular cases.  The monotone operators involved are assumed
to be subdifferentials, hence covering the context of numerical optimization.
Our assumptions are discussed at length in this scenario.  The case when the
monotone operators are linear maps is addressed as well.  Section~\ref{sec-apt}
analyzes the dynamical behavior of the iterates.  It is shown that the
piecewise linear interpolation of the iterates converges narrowly, uniformly on
compact sets, to a solution to the DI.  The result, which has its own interest,
is the first key argument to establish the main Theorem of
Section~\ref{sec-results}.  The second argument is provided in
Section~\ref{sec-clust}, where we characterize the cluster points of the
invariant mesures (indexed by the step size) of the Markov chain formed by the
iterates.  The appendices~\ref{anx:case-studies} and~\ref{anx:apt} are devoted
to the proofs relative to Sections~\ref{sec-tightness} and~\ref{sec-apt}
respectively.

\section{Background and problem statement}
\label{sec-bckgrd}

\subsection{Basic facts on maximal monotone operators}
\label{basic}

We start by recalling some basic facts related with the maximal monotone
operators on $E$ and with their associated differential inclusions. These facts
will be used in the proofs without mention.  For more details, the reader is
referred to the treatises \cite{bre-livre73}, \cite{aub-cel-(livre)84}, or
\cite{bau-com-livre11}, or to the tutorial paper \cite{pey-sor-10}. 

Consider a set valued mapping $\sA : E \rightrightarrows E$, \emph{i.e.}, 
for each $x\in E$, $\sA(x)$ is a subset of $E$. The domain and the graph of 
$\sA$ are the respective subsets of $E$ and $E\times E$ defined as 
$\dom(\sA) \eqdef \{ x \in E \, : \, \sA(x) \neq \emptyset \}$, 
and $\graph(\sA) \eqdef \{ (x,y) \in E\times E \, : \, y \in \sA(x) \}$.  
The operator $\sA$ is proper if $\dom(\sA)\neq\emptyset$. 
The operator $\sA$ is said to be monotone if 
$\forall x,x'\in \dom(\sA)$, $\forall y \in \sA(x), \forall y' \in \sA(x')$, it
holds that $\ps{y-y',x-x'}\geq 0$.  A proper monotone operator
$\sA$ is said maximal if its graph $\graph(\sA)$ is a maximal element
in the inclusion ordering among graphs of monotone operators. 

Denote by $I$ the identity operator, and by $\sA^{-1}$ the inverse of the
operator $\sA$, defined by the fact that $(x,y) \in \graph(\sA^{-1})
\Leftrightarrow (y,x) \in \graph(\sA)$.  It is well known that $\sA$ belongs to
the set $\maxmon$ of the maximal monotone operators on $E$ if
and only if, for all $\gamma > 0$, the so called resolvent operator $\sJ_\gamma
\eqdef ( I + \gamma \sA )^{-1}$ is a contraction defined on the whole space $ E$
(in particular, $\sJ_\gamma$ is single valued). We also know that when $\sA \in
\maxmon$, the closure $\cl({\dom}(\sA))$ of $\dom(\sA)$ is convex, and 
$\lim_{\gamma\to 0} \sJ_\gamma(x) = \Pi_{\cl(\dom(\sA))}(x)$, where 
$\Pi_S$ is the projector on the closed convex set $S$. 
It holds that $\sA(x)$ is closed and convex
for all $x \in \dom(\sA)$. We can therefore put $\sA_0(x) = \Pi_{\sA(x)}(0)$,
in other words, $\sA_0(x)$ is the minimum norm element of $\sA(x)$. Of
importance is the so called Yosida regularization of $\sA$ for $\gamma > 0$,
defined as the single-valued operator $\sA_\gamma = ( I - \sJ_\gamma ) /
\gamma$.  This is a $1/\gamma$-Lipschitz operator on $ E$ that satisfies 
$\sA_\gamma(x) \to \sA_0(x)$ and $\|\sA_\gamma(x)\| \uparrow \| \sA_0(x) \|$
for all $x\in \dom(\sA)$.  One can also check that $\sA_\gamma(x) \in
\sA(\sJ_\gamma(x))$ for all $x\in E$. 

A typical maximal monotone operator is the subdifferential $\partial f$ of a
function $f\in \Gamma_0$, the set of proper, convex, and lsc functions on $E$.
In this case, the resolvent $( I + \gamma\partial f)^{-1}$ for $\gamma > 0$ is
the well known proximity operator of $\gamma f$, and is denoted as
$\prox_{\gamma f}$. The Yosida regularization of $\partial f$ for $\gamma > 0$
coincides with the gradient of the so-called Moreau's envelope 
$f_\gamma(x) \eqdef  \min_w ( f(w) + \| w - x \|^2 / (2\gamma) )$ of $f$. \\

\subsection{Set valued integrals and random maximal monotone operators} 
\label{mon} 

Let $(\Xi, \mcG, \mu)$ be a probability space where the $\sigma$-field $\mcG$
is $\mu$-complete. For any Euclidean space $E$, denote as ${\mcB}(E)$ the Borel field
of $E$, and let $F : \Xi \rightrightarrows E$ be a set valued function
such that $F(s)$ is a closed set for each $s \in \Xi$.  The function $F$ is
said \emph{measurable} if $\{ s \, : \, F(s) \cap H \neq \emptyset \} \in
{\mcG}$ for any set $H \in {\mcB}(E)$.  An equivalent definition for the
mesurability of $F$ requires that the domain 
$\dom(F) \eqdef \{ s \in \Xi \, : \, F(s) \neq \emptyset \}$ of $F$ 
belongs to $\mcG$, and that there exists a sequence of measurable functions
$\varphi_n : \dom(F) \to E$ such that 
$F(s) = \cl{\{\varphi_n(s) \}_n}$ for all $s \in \dom(F)$ 
\cite[Chap.~3]{cas-val77} \cite{hia-um-77}. 

Assume now that $F$ is measurable and that $\mu(\dom(F)) = 1$. Given
$1 \leq p < \infty$, let ${\mathcal L}^p(\Xi, {\mcG}, \mu; E)$ be the 
Banach space of the ${\mcG}$-measurable functions $\varphi : \Xi \to E$ such that 
$\int \| \varphi \|^p d\mu < \infty$, and let 
\[
\Selec^p_F \eqdef 
\{ \varphi \in {\mathcal L}^p(\Xi, {\mcG}, \mu; E) \, : \, 
\varphi(s) \in F(s) \ \mu-\text{a.e.} \} \, .
\]
If $\Selec^1_F \neq \emptyset$, the function $F$ is said integrable. 
The \emph{selection integral} \cite{molchanov2006theory} of $F$ is the set 
\[
\int F d\mu \eqdef \cl{\left\{ \int_\Xi \varphi d\mu \ : \ 
  \varphi \in \Selec^1_F \right\}} .
\]
Now, consider a function $A : \Xi \to \maxmon$.  By the maximality of $A(s)$,
the graph $\graph(A(s))$ of $A(s)$ is a closed subset of $ E \times  E$
\cite{bre-livre73}.  For any $\gamma > 0$, denote by $J_{\gamma}(s, \cdot)
\eqdef  ( I + \gamma A(s) )^{-1}(\cdot)$  the resolvent of $A(s)$.  Assume that
the function $s \mapsto \graph(A(s))$ is measurable as a closed set-valued
$\Xi\rightrightarrows  E \times  E$ function.  As shown in
\cite[Ch.~2]{att-79}, this is equivalent to saying that the function $s \mapsto
J_{\gamma}(s, x)$ is measurable from $\Xi$ to $ E$ for any $\gamma > 0$ and any
$x \in  E$.  Observe that since $J_{\gamma}(s, x)$ is measurable in $s$
and continuous in $x$ (being non expansive), $J_\gamma : \Xi \times  E \to  E$
is ${\mcG} \otimes {\mcB(E)} / {\mcB(E)}$ measurable by Carath\'eodory's
theorem. 
Denoting by $D(s)$ the domain of $A(s)$, the measurability of 
$s \mapsto \graph(A(s))$ implies that the set-valued function 
$s \mapsto \cl(D(s))$ is measurable, which implies that the function
$s\mapsto d(x, D(s))$ is measurable for each $x\in E$, where $d(x,S)$ is the
distance between the point $x$ and the set $S$. 
Denoting as $A(s,x)$ the image of $x$ by the operator $A(s)$, 
the measurability of the set valued function $s\mapsto A(s,x)$ for each
$x\in E$ is another consequence of the measurability of 
$s \mapsto \graph(A(s))$. In particular, the function 
$s\mapsto A_0(s,x)$ is measurable for each $x\in E$, where 
$A_0(s,x) \eqdef \Pi_{A(s,x)}(0)$. 

The essential intersection $\mD$ of the domains $D(s)$ is defined as 
\cite{hu-these-77} 
\[
\mD \eqdef \bigcup_{G \in {\mcG} : \mu(G) = 0} \ 
\bigcap_{s \in \Xi \setminus G} D(s) \, , 
\]
in other words, 
$x \in \mD  \ \Leftrightarrow \ \mu(\{ s \, : \, x \in D(s) \}) = 1$.  Let us
assume that $\mD \neq \emptyset$, and  that the set-valued mapping $A(\cdot,x)$
is integrable for each $x \in \mD$. For all $x\in\mD$, we can define 
\[
\mA(x) \eqdef \int_\Xi A(s, x) \, \mu(ds) \, . 
\]
One can immediately see that the operator $\mA : \mD \rightrightarrows  E$ 
so defined is a monotone operator.

\subsection{Differential inclusion involving maximal monotone operators}

We now turn to the differential inclusions induced by maximal monotone
operators. Given $\sA \in \maxmon$ and $x_0 \in\dom(\sA)$, the DI 
$\dot \sx(t) \in - \sA(\sx(t))$ on $\bR_+$ with $\sx(0) = x_0$ has a unique
solution, \emph{i.e.}, a unique absolutely continuous mapping 
$\sx : \bR_+ \to E$ such that $\sx(0) = x_0$, and 
$\dot \sx(t) \in - \sA(\sx(t))$ for almost all $t > 0$. 

Consider the map $\Phi : \dom(\sA) \times \bR_+ \to \dom(\sA)$, 
$(x_0, t) \mapsto \sx(t)$ where $\sx(t)$ is the DI solution with initial value
$x_0$.  Then, $\Phi$ satisfies $\| \Phi(x,t) - \Phi(y,t)  \| \leq \| x - y \|$
for all $t\geq 0$ and all $x,y\in\dom(\sA)$. Since $E$ is complete, $\Phi$ can be extended
to a map from $\cl(\dom(\sA)) \times \bR_+$ to $\cl(\dom(\sA))$. 
This extension that we still denote as $\Phi$ is a semiflow on 
$\cl(\dom(\sA)) \times \bR_+$, being a continuous 
$\cl(\dom(\sA)) \times \bR_+ \to \cl(\dom(\sA))$ function 
satisfying $\Phi(\cdot, 0) = I$, and $\Phi(x, t+s) = \Phi(\Phi(x,s),t)$ for
each $x \in \cl(\dom(\sA))$, and $t,s\geq 0$. 

The set of zeros $Z(\sA) \eqdef  \{ x \in \dom(\sA)\, : \, 0 \in \sA(x) \}$ of
$\sA$ is a closed convex set which coincides with the set of equilibrium points 
$\{ x \in \cl(\dom(\sA)) \, : \, \forall t \geq 0, \Phi(x,t) = x \}$ of
$\Phi$.  The trajectories $\Phi(x,\cdot)$ of the semiflow do not necessarily
converge to $Z(\sA)$ (see \cite{pey-sor-10} for a counterexample). However, the
ergodic theorem for the semiflows generated by the elements of $\maxmon$ states
that if $Z(\sA) \neq \emptyset$, then for each $x \in \cl(\dom(\sA))$, 
the averaged function 
\[
\begin{array}{cccccl}
\overline\Phi &:& \cl(\dom(\sA)) \times \bR_+ &\longrightarrow& 
                                          \cl(\dom(\sA)) \\ 
& &(x,t) &\longmapsto& \displaystyle{\frac 1t \int_0^t \Phi(x,s) \, ds} 
\end{array} 
\]
(with $\overline{\Phi(\cdot, 0)} = \Phi(\cdot, 0)$), converges to an element 
of $Z(\sA)$ as $t\to\infty$. 
The convergence of the trajectories of the semiflow itself to an element of 
$Z(\sA)$ is ensured when $\sA$ is demipositive \cite{bru-75}. An operator 
$\sA\in\maxmon$ is said demipositive if there exists $w\in Z(\sA)$ such that 
for every sequence $( (u_n,v_n) \in \graph(\sA))$ such that $(u_n)$ converges 
to $u$, and such that $(v_n)$ is bounded, 
\[
\ps{u_n - w, v_n} \xrightarrow[n\to\infty]{} 0 \quad \Rightarrow \quad
u \in Z(\sA) .
\]
Under this condition and if $Z(\sA) \neq \emptyset$, then for all 
$x \in \cl(\dom(\sA))$, $\Phi(x,t)$ converges as $t\to\infty$ to an 
element of $Z(\sA)$. 

We recall some of the most important notions related with the dynamical
behavior of the semiflow $\Phi$. Denote as $\cM(E)$ the space of probability
measures on $E$ equipped with its Borel $\sigma$-field $\mcB(E)$. An element 
$\pi \in \cM(E)$ is called an
invariant measure for $\Phi$ if $\pi = \pi \Phi(\cdot, t)^{-1}$ for every 
$t > 0$. The set of invariant measures for $\Phi$ will be denoted $\cI(\Phi)$.
The limit set of the trajectory $\Phi(x,\cdot)$ of the semiflow $\Phi$ 
starting at $x$ is the set
\[
L_{\Phi(x,\cdot)} \eqdef \bigcap_{t\geq 0} \cl\left(\Phi(x, [t,\infty))\right)
\]
of the limits of the convergent subsequences $(\Phi(x, t_k))_k$ as 
$t_k\to\infty$. A point $x\in \cl(\dom\sA)$ is said recurrent if 
$x \in L_{\Phi(x,\cdot)}$. The Birkhoff center $\text{BC}_\Phi$ of $\Phi$ is 
\[
\text{BC}_\Phi \eqdef \cl{\{ x\in \cl(\dom\sA) \, : \, 
     x \in L_{\Phi(x,\cdot)} \}} \, , 
\]
\emph{i.e.}, the closure of the set of recurrent points of $\Phi$. 
The celebrated Poincar\'e's recurrence theorem 
\cite[Th.~II.6.4 and Cor.~II.6.5]{devries93} says that the support of any 
$\pi\in\cI(\Phi)$ is a subset of $\text{BC}_\Phi$.

\begin{proposition}
\label{prop-Z(A)}
Assume that $Z(\sA) \neq\emptyset$, and let $\pi \in \cI(\Phi)$. If $\sA$ is 
demipositive, then $\support(\pi) \subset Z(\sA)$. 
If $\pi$ has a first moment, then, whether $\sA$ is demipositive or not, 
\[
\int x\, \pi(dx) \in Z(\sA) \, . 
\]
\end{proposition} 
\begin{proof} 
When $\sA$ is demipositive, $Z(\sA)$ coincides straightforwardly with
$\text{BC}_\Phi$, and the first inclusion follows from Poincar\'e's recurrence 
theorem. 

To show the second result, we start by proving that 
$\{\overline \Phi(\cdot, t)\, : \, t>0\}$ is uniformly integrable as a 
family of random variables in $(E,\mcB(E),\pi)$. Let $\varepsilon>0$. Since 
the family $\{\Phi(\cdot, t)\, : \, t\geq 0\}$ is identically distributed, 
it is uniformly integrable, thus, there exists $\eta_\varepsilon>0$ such that 
$\sup_t \int_S\|\Phi(x,t)\|\, \pi(dx) \leq \varepsilon$ for all $S\in\mcB(E)$ 
satisfying $\pi(S)\leq \eta_\varepsilon$. By Tonelli's theorem, 
\[
\sup_{t>0} \int_S \|\overline \Phi(x,t)\| \, \pi(dx) \leq 
\sup_{t>0} \frac 1t \int_0^t \int_S \|\Phi(x,s)\| \, \pi(dx)ds 
\leq \varepsilon\, , 
\] 
which shows that, indeed, $\{\overline \Phi(\cdot,t):t>0\}$ is uniformly 
integrable \cite[Prop.~II-5-2]{neveu1965bases}. By the ergodic theorem
for semiflows generated by elements of $\maxmon$, there exists a 
function $f : \cl(\dom\sA) \to Z(\sA)$ such that 
$\overline \Phi(\cdot,t) \to f$ as $t\to\infty$. Since 
\[
\int x \, \pi(dx) = \int \overline \Phi(x,t) \, \pi(dx) \quad 
\text{for all } t \geq 0 \, , 
\]
we can make $t\to\infty$ and use the uniform integrability of 
$\{\overline \Phi(\cdot,t):t>0\}$ to obtain that 
$\int \| f \| \, d\pi < \infty$, and 
$\int x \, \pi(dx) = \int f(x) \, \pi(dx)$. The result follows from the 
closed convexity of $Z(\sA)$. 
\end{proof} 

\subsection{Presentation of the stochastic Forward-Backward algorithm} 
\label{sto-fb} 

Let $B:\Xi \times  E \to  E$ be a mapping such that $B(\cdot, x)$ is 
${\mcG}$-measurable for all $x\in E$, and $B(s,\cdot)$ is continuous 
and monotone (seen as a single-valued operator) on $ E$. By Carath\'eodory's 
theorem, $B$ is ${\mcG}\otimes {\mcB}( E)$-measurable. 
Furthermore, since $B(s,\cdot)$ is continuous on $ E$, this monotone operator
is maximal \cite[Prop.~2.4]{bre-livre73}. We also assume that the mapping $B(\cdot, x): \Xi \to  E$ is integrable 
for all $x\in E$, and we set $\mB(x) \eqdef \int B(s,x)\mu(ds)$. Note that 
$\dom \mB =  E$.  

Let $(\xi_n)$ be an i.i.d.~sequence of random variables from a probability
space $(\Omega, {\mcF}, \bP)$ to $(\Xi, \mcG)$ with the distribution $\mu$. Let
$x_0$ be a $E$-valued random variable with probability law $\nu$, and assume
that $x_0$ and $(\xi_n)$ are independent. Starting from $x_0$, our purpose is to
study the behavior of the iterates 
\begin{equation}
\label{fb-sto} 
x_{n+1} = 
J_\gamma (\xi_{n+1}, x_n - \gamma B(\xi_{n+1}, x_n) ) , 
\quad n \in \bN \, , 
\end{equation} 
for a given $\gamma > 0$, where we recall the notation 
$J_\gamma(s,\cdot)\eqdef (I+\gamma A(s))^{-1}(\cdot)$ for every $s\in \Xi$.  

In the deterministic case where the functions $A(s,\cdot)$ and $B(s,\cdot)$ are
replaced with deterministic maximal monotone operators $\sA(\cdot)$ and
$\sB(\cdot)$, with $\sB$ still being assumed single-valued with 
$\dom(\sB) = E$, the algorithm coincides with the well-known Forward-Backward
algorithm \eqref{fb-deterministe}.  Assuming that $\sB$ is so-called cocoercive
and that $\gamma$ is not too large, the iterates given by
(\ref{fb-deterministe}) are known to converge to an element of $Z(\sA+\sB)$,
provided this set is not empty \cite[Th.~25.8]{bau-com-livre11}.  In the
stochastic case who is of interest here, this convergence does not hold in
general.  Nonetheless, we shall show below that in the long run, the
probability that the iterates or their empirical means stay away of
$Z(\mA+\mB)$ is small when $\gamma$ is close to zero.

\section{Assumptions and main results} 
\label{sec-results}

We first observe that the process $(x_n)$ described by Eq.~\eqref{fb-sto} is
a homogeneous Markov chain whose transition kernel $P_\gamma$ is defined by the
identity
\begin{equation}
\label{Pgamma} 
P_\gamma(x,f) = \int f( J_\gamma(s, x - \gamma B(s,x))) \, \mu(ds) \, , 
\end{equation} 
valid for each measurable and positive function $f$. 
The kernel $P_\gamma$ and the initial measure $\nu$ determine completely the 
probability distribution of the process $(x_n)$, seen as a 
$(\Omega, \mcF) \to (E^\bN, \mcB(E)^{\otimes \bN})$ random variable. We
shall denote this probability distribution on $(E^\bN, \mcB(E)^{\otimes \bN})$ 
as $\bP^{\nu, \gamma}$. We denote by $\bE^{\nu,\gamma}$ the corresponding 
expectation. When $\nu=\delta_a$ for some $a\in E$, we shall prefer the 
notations $\bP^{a,\gamma}$ and $\bE^{a,\gamma}$ to $\bP^{\delta_a,\gamma}$ and 
$\bE^{\delta_a,\gamma}$. From nom on, $(x_n)$ will denote the canonical process on the canonical space $(E^\bN, \mcB(E)^{\otimes \bN})$.

We denote as $\mcF_n$ the sub-$\sigma$-field of $\mcF$ generated by the family 
$\{ x_0, \{\xi_{k}^\gamma:1\leq k\leq n\}\}$, and we write 
$\bE_n[\cdot] = \bE[\cdot \, | \, \mcF_n]$ for $n\in \bN$. 

In the remainder of the paper, $C$ will always denote a positive constant
that does not depend on the time $n$ nor on $\gamma$. This constant may change
from a line of calculation to another. In all our derivations, $\gamma$ will
lie in the interval $(0,\gamma_0]$ where $\gamma_0$ is a fixed constant 
which is chosen as small as needed. 

\subsection{Assumptions}

\begin{assumption} 
\label{A0bnd} 
For every compact set $\cK\subset E$, there exists $\varepsilon>0$ 
such that
\[
\sup_{x\in\cK \cap \mD} 
\int \| A_0(s, x) \|^{1+\varepsilon} \, \mu(ds) 
< \infty .  
\]
\end{assumption} 

\begin{assumption} 
\label{Amax} The monotone operator $\mA$ is maximal. 
\end{assumption} 

\begin{assumption} 
\label{BBnd} 
For every compact set $\cK\subset E$, there exists $\varepsilon > 0$ such that 
\[
\sup_{x\in\cK} \int \| B(s, x) \|^{1+\varepsilon} \, \mu(ds) 
  < \infty \, . 
\]
\end{assumption} 

The next assumption will mainly lead to the tightness of the invariant
measures mentioned in the introduction. 

We know that a point $x_\star$ is an element of $Z(\mA+\mB)$ if
there exists $\varphi \in \Selec^1_{A(\cdot, x_\star)}$ such that 
$\int \varphi(s) \, \mu(ds) + \int B(s,x_\star) \, \mu(ds) =0$.  
When $B(\cdot,x_\star) \in {\mathcal L}^{2}(\Xi, {\mcG}, \mu;  E)$, and 
when the above function $\varphi$ can be chosen in 
${\mathcal L}^{2}(\Xi, {\mcG}, \mu;  E)$, we say that such a zero admits a 
${\mathcal L}^{2}$ representation $(\varphi, B)$. In this case, we define 
\begin{align}
\psi_\gamma(x) &:= 
 \int \Bigl\{ \ps{A_\gamma(s,x - \gamma B(s,x)) -\varphi(s), 
  J_\gamma(s, x - \gamma B(s,x)) - x_\star} \nonumber \\
 & 
\ \ \ \ \ \ \ \ \ \ \ \ \ \ \ \ \ \ \ \ \ \ \ \ 
\ \ \ \ \ \ \ \ \ \ \ \ \ \ \ \ \ 
  + \ps{B(s, x) - B(s,x_\star), x - x_\star}  \Bigr\} \, \mu(ds) \nonumber \\ 
&\phantom{=} 
 + \gamma \int \| A_\gamma(s, x - \gamma B(s,x)) \|^2 \mu(ds) 
 - 6\gamma \int \| B(s,x) - B(s,x_\star) \|^2 \mu(ds) \, , 
\label{psig} 
\end{align} 
where 
\[
A_\gamma(s,x) \eqdef \frac{x - J_\gamma(s,x)}{\gamma} 
\] 
is the Yosida regularization of $A(s,x)$ for $\gamma > 0$. 

\begin{assumption} 
\label{psi-increase} 
There exists $x_\star \in Z(\mA+\mB)$ admitting a ${\mathcal L}^{2}$ 
representation $(\varphi, B)$. The function 
$\Psi(x) \eqdef \inf_{\gamma \in(0,\gamma_0]}\psi_\gamma(x)$ satisfies one of 
the following properties: 

\begin{enumerate}[label=(\alph*)] 


\item\label{psi-sous-lin} 
$\displaystyle{\liminf_{\|x\|\to\infty} \frac{\Psi(x)}{\| x\|}  > 0}$. 

\item\label{psi-linear} 
$\displaystyle{\frac{\Psi(x)}{\| x\|} \xrightarrow[\|x\|\to\infty]{} \infty}$.  

\item\label{psi-quad} 
$\displaystyle{\liminf_{\|x\|\to\infty} \frac{\Psi(x)}{\| x\|^2}  > 0}$. 

\end{enumerate} 
\end{assumption} 

Let us comment these assumptions. 

Assumptions \ref{A0bnd} and \ref{BBnd} are moment assumptions on $A_0(s,x)$ 
and $B(s,x)$ that are usually easy to check. Assumption  \ref{A0bnd} implies that for
every $x\in \mD$, $A_0(\,.\,,x)$ is integrable. Therefore, $A(\,.\,,x)$ is integrable.
This implies that the domain of the selection integral $\mA$ coincides with $\mD$.

Conditions where Assumption~\ref{Amax} are satisfied can be found in
\cite[Chap.~II.6]{bre-livre73} in the case where $\mu$ has a finite support, 
and in \cite[Prop.~3.1]{bia-hac-16} in other cases. When 
$A(s)$ is the subdifferential of a function $g(s,\cdot)$ belonging to 
$\Gamma_0$, the maximality of $\mA$ is established if we can exchange 
the expectation of $g(\xi_1,x)$ w.r.t.~$\xi_1$ with the subdifferentiation 
w.r.t.~$x$, in which case $\mA$ would be equal to $\partial G$, where 
$G(x) = \int g(s,x) \, \mu(ds)$. This problem is dealt with 
in~\cite{wal-wet-69} (see also Sec.~\ref{subdif} below). 

The first role of Assumption~\ref{psi-increase} is to ensure the tightness of 
the invariant measures of the kernels $P_\gamma$, as mentioned in the
introduction. Beyond the tightness, this assumption controls the asymptotic
behavior of functionals of the iterates with a prescribed growth condition at
infinity. Assumption~\ref{psi-increase} will be specified and commented at
length in Section~\ref{sec-tightness}. 
\medskip

Regarding the domains of the operators $A(s)$, two cases will be considered,
according to whether these domains vary with $s$ or not.
We shall name these two
cases the ``common domain'' case and the  ``different domains'' case 
respectively.
In the common domain case, our assumption is therefore:
\begin{assumption}[Common domain case]
  \label{common}
The set-valued function $s\mapsto D(s)$ is $\mu$-almost
everywhere constant.
\end{assumption}
In the common domain case, Assumptions \ref{A0bnd}--\ref{psi-increase} 
will be sufficient to state our results, whereas in the different domains case,
three supplementary assumptions will be needed:
\begin{assumption} [Different domains case]
\label{linreg} 
$\displaystyle{
\forall x\in E, \ \int d(x,D(s))^2 \, \mu(ds) \geq C \bs d(x)^2}$, 
where $\bs d(\cdot)$ is the distance function to~$\mD$. 
\end{assumption} 

\begin{assumption} [Different domains case]
\label{JBnd-dif} 
For every compact set $\cK\subset E$, there exists $\varepsilon > 0$ such that 
\[
\sup_{\gamma\in (0, \gamma_0], x \in \cK} 
\frac 1{\gamma^{1+\varepsilon}}
\int \| J_\gamma(s, x) - \Pi_{\cl(D(s))}(x) \|^{1+\varepsilon}  
 \, \mu(ds)  < \infty \, . 
\]
\end{assumption} 

\begin{assumption} [Different domains case]
\label{JBgrow} 
For all $\gamma\in (0,\gamma_0]$ and all $x\in E$, 
\[
\int \left( \frac{\| J_\gamma(s, x) - \Pi_{\cl(D(s))}(x) \|}{\gamma}
 + \| B(s, x) \| \right) \, \mu(ds)  
 \leq C ( 1 + \psi_\gamma(x)) \, . 
\]

\end{assumption} 

Assumption~\ref{linreg} is rather mild, and is easy to illustrate in the
case where $\mu$ is a finite sum of Dirac measures.  Following
\cite{bauschke1999strong}, we say that a finite collection of closed and convex
subsets $\{\mC_1,\dots,\mC_m\}$ over $E$ is \emph{linearly regular} if 
there exists $\kappa>0$ such that for every $x$, 
\[ 
\max_{i=1\dots m} d(x,\mC_i)\geq \kappa d(x,\mC) ,  
  \quad \text{where}\ \mC=\bigcap_{i=1}^m \mC_i\, , 
\] 
and where implicitly $\mC\neq \emptyset$. Sufficient conditions for a 
collection of sets to satisfy the above condition can be found in 
\cite{bauschke1999strong} and the references therein.

We know that when $\gamma\to 0$, $J_\gamma(s, x)$ converges to 
$\Pi_{\cl(D(s))}(x)$ for each $(s,x)$. Assumptions~\ref{JBnd-dif} 
and~\ref{JBgrow} add controls on the convergence rate. The instantiations of
these assumptions in the case of the stochastic proximal gradient algorithm
will be provided in Section~\ref{subdif} below.

\subsection{Main result} 

\begin{lemma}
\label{A+B:max} 
Let Assumptions~\ref{Amax} and \ref{BBnd} hold true. Then, the monotone 
operator $\mA + \mB$ is maximal. 
\end{lemma} 
\begin{proof} 
Assumption~\ref{BBnd} implies that the monotone operator $\mB$ is
continuous on $ E$. Therefore, $\mB$ is maximal
\cite[Prop.~2.4]{bre-livre73}.  The maximality of $\mA + \mB$ follows, since
$\mA$ is maximal by Assumption~\ref{Amax}, and $\mB$ has a full domain
\cite[Cor.~2.7]{bre-livre73}. 
\end{proof}

Note that $\dom(\mA + \mB) = \mD$.  In the remainder of the paper, we denote as
$\Phi : \cl(\mD) \times \bR_+ \to \cl(\mD)$ the semiflow produced by the DI
$\dot\sx(t) \in - (\mA + \mB)(\sx(t))$. Recall that $\cI(\Phi)$ is the set of
invariant measures for the semiflow $\Phi$. 

We also write 
\[
\bar x_n \eqdef \frac{1}{n+1} \sum_{k=0}^n x_k \, .
\]

We now state our main theorem. 
\begin{theorem}
\label{the:CV}
Let Assumptions \ref{A0bnd}, \ref{Amax},  \ref{BBnd}, and 
\ref{psi-increase}--\ref{psi-sous-lin} be satisfied.
Moreover, assume that either Assumption~\ref{common} 
or Assumptions \ref{linreg}--\ref{JBgrow} are satisfied.

Then, $\cI(\Phi) \neq\emptyset$. Let $\nu\in\cM(E)$ be with a 
finite second moment, and let 
$\mU \eqdef \bigcup_{\pi \in \cI(\Phi)} \support(\pi)$. Then, for all 
$\varepsilon > 0$, 
\begin{equation}
\label{cvg:support}
\limsup_{n\to\infty} 
\frac 1{n+1}\sum_{k=0}^n \bP^{\nu,\gamma}( d(x_k,\mU)>\varepsilon)
 \xrightarrow[\gamma\to 0]{}0\,. 
\end{equation}
In particular, if the operator $\mA + \mB$ is demipositive, then 
\begin{equation}
\label{cvg-1/2pos} 
\limsup_{n\to\infty} 
\frac 1{n+1}\sum_{k=0}^n 
 \bP^{\nu,\gamma}\left( d(x_k,Z(\mA+\mB))>\varepsilon\right)
 \xrightarrow[\gamma\to 0]{}0\,. 
\end{equation}
Moreover, the set $\{ \pi \in \cI(\Phi) \, : \, \pi(\Psi) < \infty \}$ is not
empty. Let $N'\in\bN^*$, and let $f : E \to \bR^{N'}$ be continuous. Assume 
that there exists $M\geq 0$ and $\varphi:\bR^{N'} \to\bR_+$ such that 
$\lim_{\|a\|\to\infty}\varphi(a)/{\|a\|}=\infty$, and
\[
\forall a\in E,\ \varphi(f(a))\leq M(1+\Psi(a))\, . 
\]
Then, for all $n\in\bN$, $\gamma\in (0,\gamma_0]$, the r.v.
$$
F_n \eqdef \frac 1{n+1}\sum_{k=0}^n f(x_k) 
$$
is $\bP$-integrable, and satisfies for all $\varepsilon>0$,
\begin{align}
&\limsup_{n\to\infty}
\bP^{\nu,\gamma}\left(d\left(F_n,\cS_f\right)\geq \varepsilon\right)
  \xrightarrow[\gamma\to 0]{}0\, ,    \label{P:CVS} \\ 
&\limsup_{n\to\infty}\ d\left(\bE^{\nu,\gamma}(F_n),\cS_f\right)
       \xrightarrow[\gamma\to 0]{}0\,.   
\label{E:CVS} 
\end{align}
where $\cS_f\eqdef\{\pi(f)\,:\, \pi\in\cI(\Phi) \}$. 
In particular, if $f(x) = x$, and if 
Assumption \ref{psi-increase}--\ref{psi-linear} is satisfied, then 
\begin{align}
&\limsup_{n\to\infty}\ 
\bP^{\nu,\gamma}\left(d\left(\bar x_n , Z(\mA+\mB) \right)\geq 
      \varepsilon\right)\xrightarrow[\gamma\to 0]{}0\, , \label{Pbarx} \\ 
&\limsup_{n\to\infty}\ d\left(\bE^{\nu,\gamma}(\bar x_n), Z(\mA+\mB)\right)
      \xrightarrow[\gamma\to 0]{}0\,. 
\label{Ebarx} 
\end{align}
\end{theorem} 

By Lem.~\ref{A+B:max} and Prop.~\ref{prop-Z(A)}, the convergences 
\eqref{cvg-1/2pos}, \eqref{Pbarx}, and \eqref{Ebarx} are the consequences of 
\eqref{cvg:support}, \eqref{P:CVS}, and \eqref{E:CVS} respectively.  
We need to prove the latter. 

\subsection{Proof technique} 

We first observe that the Markov kernels $P_\gamma$ are Feller, \emph{i.e.},
they take the set $C_b(E)$ of the real, continuous, and bounded functions on
$E$ to $C_b(E)$. Indeed, for each $f\in C_b(E)$, Eq.~\eqref{Pgamma} shows that
$P_\gamma(\cdot, f) \in C_b(E)$ by the continuity of $J_\gamma(s,\cdot)$ and
$B(s,\cdot)$, and by dominated convergence.

For each $\gamma > 0$, we denote as 
\[
\cI(P_\gamma) \eqdef \{\pi\in\cM(E)\,:\,\pi = \pi P_\gamma\}
\]
the set of invariant probability measures of $P_\gamma$. Define the family of 
kernels $\cP \eqdef \{ P_\gamma \}_{\gamma\in(0,\gamma_0]}$, and let 
\[
\cI(\cP) \eqdef \bigcup_{\gamma\in(0,\gamma_0]}\cI(P_\gamma)
\]
be the set of distributions $\pi$ such that $\pi=\pi P_\gamma$ for at least 
one $P_\gamma$ with $\gamma\in(0,\gamma_0]$. 

The following proposition, which is valid for Feller Markov kernels, has been
proven in \cite{bia-hac-sal-(arxiv)16} in the more general context of
set-valued differential inclusions.

\begin{proposition}
\label{prop:PH}
Let $V:E\to[0,+\infty)$ and $Q:E\to[0,+\infty)$ be measurable. 
Assume that $Q(x) \to \infty$ as $\| x \| \to\infty$. 
Assume that for each $\gamma\in(0,\gamma_0]$, 
\begin{equation}
\label{eq:PH} 
P_\gamma(x, V) \leq V(x) -\alpha(\gamma) Q(x) +\beta(\gamma)\, , 
\end{equation} 
where $\alpha:(0,\gamma_0]\to(0,+\infty)$ and $\beta:(0,\gamma_0]\to\bR$ 
satisfy 
$\sup_{\gamma\in(0,\gamma_0]}\frac{\beta(\gamma)}{\alpha(\gamma)}<\infty$.
Then, the family $\cI(\cP)$ is tight. Moreover, 
$\sup_{\pi\in\cI(\cP)} \pi(Q) < \infty$. 

Assume moreover that, as $\gamma\to 0$, any cluster point of $\cI(\cP)$ is an element of
$\cI(\Phi)$. In particular, $\{\pi \in \cI(\Phi) \, : \, \pi(Q) < \infty \}$ is
not empty. Let $\nu\in\cM(E)$ s.t. $\nu(V)<\infty$.
Let $\mU \eqdef \bigcup_{\pi \in \cI(\Phi)} \support(\pi)$. Then, for all 
$\varepsilon > 0$, 
\[
\limsup_{n\to\infty} \frac 1{n+1}\sum_{k=0}^n 
\bP^{\nu,\gamma}( d(x_k,\mU)>\varepsilon)\xrightarrow[\gamma\to 0]{}0\,. 
\]
Let $N'\in\bN^*$ and $f : E\to \bR^{N'}$ be continuous. Assume that there 
exists $M\geq 0$ and $\varphi:\bR^{N'}\to\bR_+$ such that 
$\lim_{\|a\|\to\infty}\varphi(a)/{\|a\|}=\infty$ and
\[
\forall a\in E,\ \varphi(f(a))\leq M(1+Q(a))\,. 
\]
Then, for all $n\in\bN$, $\gamma\in (0,\gamma_0]$, the r.v.
$$
F_n \eqdef \frac 1{n+1}\sum_{k=0}^n f(x_k)
$$
is $\bP^{\nu,\gamma}$-integrable, and satisfies for all $\varepsilon>0$,
\[
\limsup_{n\to\infty} d\left(\bE^{\nu,\gamma}( F_n)\,,\cS_f\right)
 \xrightarrow[\gamma\to 0]{}0\,, \quad \text{and} \quad 
\limsup_{n\to\infty} 
 \bP^{\nu,\gamma}\left(d\left(F_n\,,\cS_f\right)\geq \varepsilon\right)
  \xrightarrow[\gamma\to 0]{}0\, , 
\]
where $\cS_f \eqdef \{ \pi(f) \, : \, \pi \in \cI(\Phi) \}$. 
\end{proposition}
\begin{proof}
Assume that Eq.~(\ref{eq:PH}) holds. By \cite[Prop. 6.7]{bia-hac-sal-(arxiv)16}, 
  $\cI(\cP)$ is tight and $\sup_{\pi\in\cI(\cP)} \pi(Q) < \infty$, which proves the first point.
Assume moreover that, as $\gamma\to 0$, any cluster point of $\cI(\cP)$ is an element of
$\cI(\Phi)$. By the tightness of $\cI(\cP)$ and the Prokhorov theorem, such a cluster point $\pi$
exists, and satisfies $\pi(Q)<\infty$ by the first point just shown.
The rest of the proof follows \cite[Section 6.4]{bia-hac-sal-(arxiv)16} word-for-word.
\end{proof}

In order to prove Th.~\ref{the:CV}, it is enough to show that the
assumptions of Prop.~\ref{prop:PH} are satisfied. Namely, we need to 
establish \eqref{eq:PH} and to show that the cluster points of $\cI(\cP)$
as $\gamma\to 0$ are elements of $\cI(\Phi)$. 

In Sec.~\ref{sec-apt}, we show that the linearly interpolated process
constructed from the sequence $(x_n)$ converges narrowly as $\gamma\to 0$ to a
DI solution in the topology of uniform convergence on compact sets.  The main
result of this section is Th.~\ref{fb-apt}, which has its own interest.  To
prove this theorem, we establish the tightness of the linearly interpolated
process (Lem.~\ref{lem:ui}), then we show that the limit points coincide with
the DI solution (Lem.~\ref{lem:sko}--\ref{lem:cvth}).  In Sec.~\ref{sec-clust},
we start by establishing the inequality \eqref{eq:PH}, which is shown in
Lem.~\ref{fb-PH} with $Q(x) = \Psi(x)$. Using the tightness of $\cI(\cP)$ in
conjunction with Th.~\ref{fb-apt}, Lem~\ref{inv-mon} shows that the cluster
points of $\cI(\cP)$ are elements of $\cI(\Phi)$. In the different domains
case, this lemma requires that the invariant measures of $P_\gamma$ put most of
their weights in a thickening of the domain $\mD$ of order $\gamma$.  This fact
is established by Lem.~\ref{thickening}.

\section{Case studies - Tightness of the invariant measures} 
\label{sec-tightness}

Before proving the main results, we first address 
three important cases: the case of the random
proximal gradient algorithm, the case where $A(s)$ is an affine monotone
operator and $B(s) = 0$, and the case where $\mD$ is bounded.  The main
problem is to ensure that one of the cases of Assumption~\ref{psi-increase} is
verified. We close the section with a general condition ensuring that
Assumption~\ref{psi-increase}--\ref{psi-sous-lin} is verified. 
The proofs are postponed to Appendix~\ref{anx:case-studies}.

\subsection{A random proximal gradient algorithm} 
\label{subdif}
Let $(\Sigma, \mcA, \zeta)$ be a probability space, where $\mcA$ is
$\zeta$-complete.  Denoting as $\epi$ the epigraph of a function, a function 
$h : \Sigma \times  E \to (-\infty, \infty]$ is called a convex normal
integrand~\cite{roc-69(mes)} if the set-valued mapping 
$s \mapsto \epi h(s, \cdot)$ is closed-valued and measurable, and if
$h(s,\cdot)$ is convex. To simplify the presentation, we furthermore assume
that $h$ is finite everywhere, noting that the results can be extended to the
case where $h$ can take the value $\infty$.  Observe that the set-valued
function $s \mapsto \partial h(s,\cdot)$ is a measurable $\Sigma\to\maxmon$
function in the sense of Section~\ref{mon} \cite{att-79} (in all what follows,
the subdifferential or the gradient of a function in $(s,x)$ will be meant to
be taken w.r.t.~$x$).  Assume that $\int | h(s,x) | \zeta(ds) < \infty$ for all
$x\in E$, and consider the convex function 
$H(x) \eqdef \int h(s,x) \, \zeta(ds)$
defined on $E$.  By \emph{e.g.}, \cite[page~179]{roc-wet-82}, 
$\partial H(x) = \int \partial h(s,x) \, \zeta(ds)$. 

Let $f : \Sigma\times E \to \bR$ be such that $f(\cdot, x)$ is
$\mcA$-measurable for all $x\in E$, and $f(s,\cdot)$ is convex and continuously
differentiable for all $s\in \Sigma$. Moreover, assume that 
$\int | f(s,x) | \, \zeta(ds) < \infty$ for all $x\in E$, and define the 
function $F(x) \eqdef \int f(s,x) \, \zeta(ds)$ on $E$. This function is 
differentiable with $\nabla F(x) = \int \nabla f(s,x) \, \zeta(ds)$. 

Finally, given $m\in \bN^*$, let $\{\mC_1,\dots,\mC_m\}$ be a collection of 
closed and convex subsets of $E$. We assume that 
$\bigcap_{i=1}^m \relint(\mC_i) \neq \emptyset$, where $\relint$ is the 
relative interior of a set. 

Our purpose is to approximatively solve the optimization problem
\begin{equation} 
\label{csrt-min} 
\min_{x\in \mC} F(x) + H(x), \quad \mC \eqdef \bigcap_{i=1}^m \mC_i \, 
\end{equation} 
whether the minimum is attained. 
Let $(u_n)$ be an iid sequence on $\Sigma$ with the probability measure 
$\zeta$. 
Let $(I_n)$ be an iid sequence on $\{0, 1, \ldots, m \}$ with the probability 
measure $\alpha$ such that $\alpha(k) = \bP(I_1=k) > 0$ for each $k$. Assume
that $(I_n)$ and $(u_n)$ are independent. In order to solve the 
problem~\eqref{csrt-min}, we consider the iterates 
\begin{equation} 
\label{sto-gra-prox} 
x_{n+1} = \left\{\begin{array}{ll} 
\prox_{\alpha(0)^{-1}\gamma h(u_{n+1}, \cdot)}
   (x_n - \gamma \nabla f(u_{n+1}, x_n)) & \text{if } I_{n+1} = 0, \\
\Pi_{\mC_{I_{n+1}}}(x_n - \gamma \nabla f(u_{n+1}, x_n)) & \text{otherwise}, 
\end{array}\right. 
\end{equation} 
for $\gamma > 0$. 
This problem can be cast in the general framework of the stochastic proximal
gradient algorithm presented in the introduction. On the space  
$\Xi \eqdef \Sigma \times \{0,\ldots, m\}$, define the iid random variables 
$\xi_n \eqdef(u_n, I_n)$ with the measure $\mu \eqdef \zeta\otimes\alpha$. 
Denoting as $\iota_S$ the indicator function of the set $S$, let 
$g : \Xi \times E \to (-\infty, \infty]$ be defined as 
\[
g(s, x) \eqdef \left\{\begin{array}{ll} \alpha(0)^{-1} h(u,x) &
\text{if } i = 0, \\
\iota_{\mC_i}(x) & \text{otherwise}, 
\end{array}\right. 
\]
where $s = (u, i)$. Then, Problem~\eqref{csrt-min} is equivalent to 
minimizing the sum $F(x) + G(x)$, where 
\[
G(x) \eqdef \int g(s,x)\, \mu(ds) = 
\sum_{k=1}^m \iota_{\mC_k}(x) + H(x) \, . 
\]
It is furthermore clear that the algorithm~\eqref{sto-gra-prox} is the instance
of the general algorithm~\eqref{fb-sto} that corresponds to 
$A(s) = \partial g(s,\cdot)$ and $B(s) = \nabla f(u,\cdot)$ for $s = (u,i)$.
With our assumptions, the qualification conditions hold, and the three sets
$\arg\min(F+G)$, $Z(\partial G + \nabla F)$, and $Z(\mA+\mB)$ coincide. 

Before going further, we recall some well known facts regarding the 
coercive functions belonging to $\Gamma_0$. A function $q\in\Gamma_0$ is
said coercive if $\lim_{\|x\|\to\infty} q(x) = \infty$. It is 
said supercoercive if $\lim_{\|x\|\to\infty} q(x)/\| x \| = \infty$. The three following 
conditions are equivalent: i) $q$ is coercive, ii) there exists $a\in\bR$ such
that the level set $\lev_{\leq a} q$ is non empty and compact, 
iii) $\lim\inf_{\| x\|\to\infty} q(x) / \| x \| > 0$ 
(see \emph{e.g.}, \cite[Prop.~11.11 and 11.12]{bau-com-livre11} and
\cite[Prop.~1.1.5]{bor-lew-livre06}). 

The main result of this paragraph is the following: 
\begin{proposition}
\label{tight-stogra} 
Let the following hypotheses hold true: 
\begin{enumerate}[label=H{\arabic*}] 
\item\label{L2} There exists $x_\star \in Z(\partial G + \nabla F)$ admitting a 
$\mathcal L^2$ representation $(\varphi((u,i)), \nabla f(u, x_\star))$. 

\item\label{baillon} There exists $c>0$ s.t. for every $x\in E$,
$$
\int  \ps{\nabla f(s,x) - \nabla f(s,x_\star), x - x_\star} \, \zeta(ds)
\geq c \int \| f(s,x) - f(s,x_\star) \|^2  \, \zeta(ds) .
$$

\item\label{F+G-coerc} 
The function $F+G$ satisfies one of the following properties: 
\begin{enumerate}[label=(\alph*)]
\item\label{Zer-cpct} $F+G$ is coercive. 
\item\label{F+G-super} $F+G$ is supercoercive. 
\end{enumerate}

\end{enumerate} 

Then, Assumption~\ref{psi-increase}--\ref{psi-sous-lin} 
(resp., Assumption~\ref{psi-increase}--\ref{psi-linear}) holds true if 
Hypothesis~\ref{F+G-coerc}--\ref{Zer-cpct} 
(resp., Hypothesis~\ref{F+G-coerc}--\ref{F+G-super}) holds true. 
\end{proposition} 

Let us comment these hypotheses. A light condition ensuring the truth of 
Hypothesis~\ref{L2} is provided by the following lemma. 

\begin{lemma}
\label{lm:L2} 
Assume that there exists $x_\star \in Z(\partial G + \nabla F)$ satisfying the
two following conditions: 
$\int \| \nabla f(u, x_\star) \|^2 \, \zeta(du) < \infty$, and there exists
an open neighborhood $\mathcal N$ of $x_\star$ such that 
$\int h(u,x)^2 \, \zeta(du) < \infty$ for all $x \in {\mathcal N}$. Then, 
Hypothesis~\ref{L2} is verified. 
\end{lemma}

We now turn to Hypothesis~\ref{baillon}. When studying the deterministic
Forward-Backward algorithm \eqref{fb-deterministe}, it is standard to assume
that $\sB$ is cocoercive, in other words, that there exists a constant $L > 0$
such that $\ps{ \sB(x) - \sB(y), x- y} \geq L \| \sB(x) - \sB(y) \|^2$
\cite[Th.~25.8]{bau-com-livre11}. A classical case where this is satisfied is
the case where $\sB$ is the gradient of a convex differentiable function having
a $1/L$-Lipschitz continuous gradient, as is shown by the Baillon-Haddad
theorem~\cite[Cor.~18.16]{bau-com-livre11}. In our case, if we assume that
there exists a nonnegative measurable function $\beta(s)$ such that 
$\| \nabla f(s, x) - \nabla f(s, x') \| \leq \beta(s) \| x - x'\|$, then by
the Baillon-Haddad theorem, 
\[
\ps{\nabla f(s,x) - \nabla f(s,x'), x - x'} 
\geq \frac{1}{\beta(s)} \| \nabla f(s, x) - \nabla f(s, x') \|^2 \, . 
\]
Thus, one obvious case where Hypothesis~\ref{baillon} is satisfied is the 
case where $\beta(s)$ is bounded.

Using proposition \ref{tight-stogra}, we can now obtain the following corollary
to Th.~\ref{the:CV}. 

\begin{corollary}
Let Hypotheses \ref{L2}--\ref{F+G-coerc} hold true. Assume in addition the 
following hypotheses:  
\begin{enumerate}[label=C{\arabic*}] 
\item\label{h0bnd} 
For every compact set $\cK\subset E$, there exists $\varepsilon>0$ such 
that
\[
\sup_{x\in\cK \cap \mC} 
\int \| \partial h_0(u, x) \|^{1+\varepsilon} \, \zeta(du) < \infty ,   
\]
where $\partial h_0(u,\cdot)$ is the least norm element of 
$\partial h(u,\cdot)$.

\item\label{nablaf-bnd} 
For every compact set $\cK\subset E$, there 
exists $\varepsilon > 0$ such that 
\[
\sup_{x\in\cK} \int \| \nabla f (u, x) \|^{1+\varepsilon} \, \zeta(du) 
  < \infty\, . 
\] 

\item\label{regfct} The sets $\mC_1,\ldots, \mC_m$ are linearly regular. 

\item\label{JBgrow-fct}  
For all $\gamma\in (0,\gamma_0]$ and all $x\in E$, 
\[
\int \left( 
 \| \nabla h_\gamma(u, x) \| + \| \nabla f(u, x) \| \right) \zeta(du)  
 \leq C ( 1 + | F(x) + H_\gamma(x)| ) \, ,  
\] 
where $h_\gamma(u,\cdot)$ is the Moreau envelope of $h(u,\cdot)$. 
\end{enumerate} 

Then, for each probability measure $\nu$ having a finite second moment, 
\[
\limsup_{n\to\infty} 
\frac 1{n+1}\sum_{k=0}^n 
 \bP^{\nu,\gamma}\left( d(x_k, \arg\min(F+G))>\varepsilon\right)
 \xrightarrow[\gamma\to 0]{}0\,. 
\] 
Moreover, if Hypothesis~\ref{F+G-coerc}--\ref{F+G-super} is satisfied, then 
\begin{gather*} 
\limsup_{n\to\infty}\ 
\bP^{\nu,\gamma}\left(d\left(\bar x_n , \arg\min(F+G) \right)\geq 
      \varepsilon\right)\xrightarrow[\gamma\to 0]{}0, \ \text{and} \\ 
\limsup_{n\to\infty}\ d\left(\bE^{\nu,\gamma}(\bar x_n), \arg\min(F+G) \right)
      \xrightarrow[\gamma\to 0]{}0\,. 
\end{gather*} 
\end{corollary} 
\begin{proof}
With the hypotheses \ref{L2}--\ref{F+G-coerc} and 
\ref{h0bnd}--\ref{JBgrow-fct}, 
one can check that the assumptions \ref{A0bnd}--\ref{JBgrow} are 
verified. Note that $\partial G + \nabla F$ is a demipositive operator, being 
the subdifferential of a $\Gamma_0$ function having a minimizer~\cite{bru-75}. 
The results of the corollary follow from those of Th.~\ref{the:CV}. 
\end{proof}

\subsection{The case where $A(s)$ is affine} 

In all the remainder of this section, we shall focus on the validity of 
Assumption~\ref{psi-increase}. We assume that $B = 0$, and that 
\[
A(s, x) = H(s) x + d(s) , 
\]
where $H : \Xi \to \bR^{N\times N}$ and $d : \Xi \to  E$ are two
$\mcG$-measurable functions. It is easily seen that the linear operator $A(s)$
is monotone if and only if $H(s) + H(s)^T \geq 0$ in the semidefinite ordering
of matrices, a condition that we shall assume in this subsection. 
Moreover, assuming that 
\[
\int ( \| H(s) \|^2 + \| d(s) \|^2 ) \, \mu(ds) < \infty, 
\]
the operator 
\[
\mA(x) = \Bigl( \int H(s)\, \mu(ds) \Bigr) x + \int d(s) \, \mu(ds) 
\eqdef \bs H x + \bs d  
\]
exists and is a maximal monotone operator with the domain $ E$. 
When $\bs d$ belongs to the image of $\bs H$, $Z(\mA) \neq\emptyset$, and every
$x_\star \in Z(\mA)$ has a unique $\mathcal L^2$ representation 
$(\varphi(s) = H(s) x_\star + d(s), 0)$. We have the following proposition: 

\begin{proposition} 
\label{affine} 
If $\bs H + \bs H^T > 0$, then $\bs H$ is invertible, 
$Z(\mA) = \{ x_\star \}$ with $x_\star = - \bs H^{-1} \bs d$, and 
and Assumption~\ref{psi-increase}--\ref{psi-quad} is verified. 
\end{proposition}

\subsection{The case where the domain $\mD$ is bounded} 
\label{D-bounded} 

\begin{proposition} 
\label{prop-Dbnd} 
Let the following hypotheses hold true: 

\begin{enumerate}[label=H{\arabic*}] 

\item The domain $\mD$ is bounded. 

\item\label{linreg-bnd} There exists a constant $C > 0$ such that
\[
\forall x\in E, \ \int d(s,x)^2 \, \mu(ds) \geq C \bs d(x)^2 . 
\]

\item There exists $x_\star \in Z(\mA+\mB)$ admitting a
${\mathcal L}^{2}$ representation. 

\item\label{cocobnd}  
There exists $c>0$ s.t. for every $x\in E$,
For all $\gamma$ small enough, 
$$
\int  \ps{B(s, x) - B(s,x_\star), x - x_\star}  \, \mu(ds)\geq 
  c\,\int \| B(s,x) - B(s,x_\star) \|^2  \, \mu(ds)  \, .
$$
\end{enumerate} 
Then, Assumption~\ref{psi-increase}--\ref{psi-quad} is satisfied. 
\end{proposition}  

\subsection{A case where Assumption~\ref{psi-increase}--\ref{psi-sous-lin} is 
 valid}

We close this section by providing a general condition that guarantees the
validity of Assumption~\ref{psi-increase}--\ref{psi-sous-lin}. For simplicity,
we focus on the case where $B(s) = 0$, noting that the result can be easily
extended to the case where $B(s) \neq 0$ when a cocoercivity hypothesis of the
type of Prop.~\ref{prop-Dbnd}--\ref{cocobnd} is satisfied. 

We denote by $\bs S(\rho, d)$ the sphere of $E$ with center $\rho$ and radius
$d$. We also denote by $\inter S$ the interior of a set $S$. 


\begin{proposition}
\label{coer-gal} 
Assume that $B(s) = 0$, and that there exists $x_\star \in Z(\mA) \cap
\inter{\mD}$ admitting a $\mathcal L^2$ representation 
$\varphi \in \Selec^2_{A(\cdot, x_\star)}$. Assume that there exists
a set $\Sigma \in \mcG$ such that $\mD \subset \cap_{s\in\Sigma} D(s)$,
$\mu(\Sigma) > 0$, and such that for all $s\in \Sigma$, there exists 
$\delta(s) > 0$ satisfying $\bs S(\varphi(s), \delta(s)) \subset \inter\mD$, 
and 
\[
\forall x \in \bs S(\varphi(s), \delta(s)), \ 
\inf_{y\in A(s,x)} \ps{y - \varphi(s), x - x_\star} > 0 .
\]
Then, Assumption~\ref{psi-increase}--\ref{psi-sous-lin} is satisfied. 
\end{proposition} 

Note that the $\inf$ in the statement of this proposition is attained, as is revealed by the proof.

\section{Narrow convergence towards the DI solutions} 
\label{sec-apt} 

\subsection{Main result}

The set $C(\bR_+,E)$ of continuous functions from $\bR_+$ to $E$ is equipped
with the topology of uniform convergence on the compact intervals, who is known
to be compatible with the distance $\distC$ defined as
\[
\distC(\sx,\sy)\eqdef \sum_{n\in\mathbb N^*}2^{-n} 
\left(1\wedge \sup_{t\in [0,n]}\|\sx(t)-\sy(t)\|\right)\, . 
\]
For every $\gamma>0$, we introduce the measurable map 
$\sX_\gamma : (E^\bN, \mcB(E)^{\otimes \bN})\to (C(\bR_+,E),\mcB(C(\bR_+,E)))$,
defined for every $x=(x_n:n\in \bN)$ in $E^\bN$ as 
\[
  \sX_\gamma(x)\,:t \mapsto x_{\lfloor \frac t\gamma\rfloor} + (t/\gamma-\lfloor t/\gamma\rfloor)(x_{\lfloor \frac t\gamma\rfloor+1}-x_{\lfloor \frac t\gamma\rfloor})
\,.
\]
This map will be referred to as the linearly interpolated process. When 
$x = (x_n)$ is the process with the probability measure $\bP^{\nu,\gamma}$ 
defined above, the  distribution of the r.v.~$\sX_\gamma$ is 
$\bP^{\nu,\gamma}\sX_\gamma^{-1}$. 
If $S$ is a subset of $E$ and $\varepsilon>0$, we denote by 
$S_\varepsilon\eqdef\{a\in E: d(a,S) <\varepsilon\}$ the 
$\varepsilon$-neighborhood of $S$. The aim of the present section is to 
establish the following result: 
\begin{theorem}
\label{fb-apt}
Let Assumptions \ref{A0bnd}--\ref{BBnd} hold true. 
Let either Assumption~\ref{common} or Assumptions~\ref{linreg}-\ref{JBnd-dif}
hold true. Then, for every $\eta>0$, for every compact set $\cK\subset E$ s.t.  $\cK\cap \cD\neq \emptyset$,
\begin{equation}
\forall M \geq 0, 
\sup_{a\in \cK\cap \mD_{\gamma M}} \bP^{a,\gamma} \left( 
 \distC ( \sX_\gamma, \Phi(\Pi_{\cl(\mD)}(a), \cdot) ) > \eta \right) 
\xrightarrow[\gamma\to 0]{} 0 .  \label{eq:apt}
\end{equation}
\end{theorem} 
Using the Yosida regularization $A_\gamma(s,x)$ of $A(s,x)$, the
iterates~\eqref{fb-sto} can be rewritten as $x_0 = a \in \mD_{\gamma M}$ and
\begin{equation} 
\label{fb-yosida} 
x_{n+1} = x_n - \gamma B(\xi_{n+1}, x_n) 
- \gamma A_\gamma(\xi_{n+1}, x_n - \gamma B(\xi_{n+1}, x_n) ) . 
\end{equation} 
Setting $h_\gamma(s,x)\eqdef - B(s, x)-  A_\gamma(s, x - \gamma B(s,x) )$,
the iterates~\eqref{fb-sto} can be cast into the same form as the one studied in~\cite{bia-hac-sal-(arxiv)16}.
The following result, which we state here mainly for the ease of the reading, is a straightforward
consequence of \cite[Th. 5.1]{bia-hac-sal-(arxiv)16}.
\begin{proposition}
Let Assumptions \ref{A0bnd}--\ref{BBnd} hold true.  Assume moreover that for 
every $s\in \Xi$, $D(s) = E$. 
Then, Eq.~(\ref{eq:apt}) holds true.
\label{prop:apt-usc}
\end{proposition}
\begin{proof}
It is sufficient to check that the mapping $h_\gamma$ satisfies the Assumption
(RM) of \cite[Th. 5.1]{bia-hac-sal-(arxiv)16}. Assumption i) in~\cite[As. (RM)]{bia-hac-sal-(arxiv)16}
is satisfied by definition of $h_\gamma$. As $D(\cdot)$ is a constant
equal to $E$, the operator $A(s,\cdot)$ is upper semi continuous as a 
set-valued operator \cite{phe-97}. Thus, 
$H(s,\cdot)\eqdef -A(s,\cdot)-B(s,\cdot)$ is proper,
upper semi continuous with closed convex values, and $\mu$-integrable.  Hence,
the assumptions iii-iv) in \cite[As. (RM)]{bia-hac-sal-(arxiv)16} are
satisfied. Assumption v) is satisfied by the natural properties of the
semiflow induced by the maximal monotone map $\mA+\mB$, whereas Assumption vi)
in \cite[As. (RM)]{bia-hac-sal-(arxiv)16} directly follows from the present
Assumptions~\ref{A0bnd} and \ref{BBnd} and the definition of $h_\gamma$. One should finally verify Assumption ii)
in \cite[As. (RM)]{bia-hac-sal-(arxiv)16}, which states that for every
converging sequence $(u_n,\gamma_n)\to (u^\star,0)$, $h_{\gamma_n}(s,u_n)\to
H(s,u^\star)$, for every $s\in \Xi$. To this end, it is sufficient to prove 
that 
\begin{equation}
  \label{eq:cvAgamma}
  A_{\gamma_n}(s,u_n-\gamma_n B(s,u_n)) \to A(s,u^\star)\,.
\end{equation}
Choose $\varepsilon>0$.
As $A(s,\cdot)$ is upper semi continuous, there exists $\eta>0$ s.t. $\forall u$, $\|u-u^\star\|<\eta$ implies $A(s,u)\subset A(s,u^\star)_\varepsilon$.
Let $v_n\eqdef J_{\gamma_n}(s,u_n-\gamma B(s,u_n))$. 
By the triangular inequality and the non-expansiveness of $J_{\gamma_n}$,  
$$
\|v_n-u^\star\|\leq \|u_n-u^\star\|+\gamma_n\|B(s,u_n)\|+\|J_{\gamma_n}(u^\star)-u^\star\|\,,
$$
where it is clear that each of the three terms in the right hand side tends to zero. Thus, there exists $N\in \mathbb N$
s.t. $\forall n\geq N$, $\|v_n-u^\star\|\leq \eta$, which in turn implies 
$A(s,v_n)\subset A(s,u^\star)_\varepsilon$. As $A_{\gamma_n}(s,u_n-\gamma_n B(s,u_n))\in A(s,v_n)$, the convergence \eqref{eq:cvAgamma} is established. 
\end{proof}

\subsection{Proof of Th.~\ref{fb-apt}}
In the sequel, we prove Theorem~\ref{fb-apt} under the set of Assumptions~\ref{linreg}-\ref{JBnd-dif}.
The proof in the common domain case \emph{i.e.}, when Assumption~\ref{common} holds, 
is somewhat easier and follows from the same arguments.

In order to prove Theorem~\ref{fb-apt}, we just have to weaken the assumptions of Proposition~\ref{prop:apt-usc}: 
for a given $s\in \Xi$, the domain
$D(s)$ is not necessarily equal to $E$ and the monotone operator $A(s,\,.\,)$ is not necessarily upper semi continuous.
Up to these changes, the proof is similar to the proof of \cite[Th. 3.1]{bia-hac-sal-(arxiv)16}
and the modifications are in fact confined to specific steps of the proof.

Choose a compact set $\cK\subset E$ s.t. $\cK\cap \bmD\neq \emptyset$. Choose $R>0$  s.t. $\cK$ is contained in
the ball of radius $R$.
For every $x=(x_n:n\in \bN)$ in $E^\bN$, define 
$\tau_R(x)\eqdef\inf\{n\in \bN:x_n>R\}$ and
introduce the measurable mapping $C_R:E^\bN\to E^\bN$, given by
$$
C_R(x) : n\mapsto x_n\1_{n< \tau_R(x)} + x_{\tau_R(x)}\1_{n\geq \tau_R(x)}\,.
$$
Consider the image measure $\bar \bP^{a,\gamma}\eqdef \bP^{a,\gamma} B_R^{-1}$, which corresponds
to the law of the \emph{truncated} process $B_R(x)$. 
The crux of the proof consists in showing that for every $\eta>0$ and every $M>0$,
\begin{equation}
  \label{eq:apt-trk}
\sup_{a\in \cK\cap \mD_{\gamma M}} \bar \bP^{a,\gamma} \left( 
 \distC ( \sX_\gamma, \Phi(\Pi_{\cl(\mD)}(a), \cdot) ) > \eta \right) 
\xrightarrow[\gamma\to 0]{} 0 .
\end{equation}
Eq.~(\ref{eq:apt-trk}) is the counterpart of \cite[Lemma 4.3]{bia-hac-sal-(arxiv)16}. Once it has been proven,
the conclusion follows verbatim from \cite[Section 4, End of the proof]{bia-hac-sal-(arxiv)16}. 
Our aim is thus to establish Eq.~(\ref{eq:apt-trk}).
The proof follows the same steps as the proof of 
\cite[Lemma 4.3]{bia-hac-sal-(arxiv)16} up to some confined changes. Here, the 
steps of the proof which do not need any modification are recalled rather 
briefly (we refer the reader to \cite{bia-hac-sal-(arxiv)16} for the details). 
On the other hand, the parts which require an adaptation are explicitly stated 
as lemmas, whose detailed proofs are provided in Appendix~\ref{anx:apt}. 

Define 
$h_{\gamma,R}(s,a)\eqdef h_{\gamma}(s,a)\1_{\|a\|\leq R}$.
First, we recall the following decomposition, established in \cite{bia-hac-sal-(arxiv)16}:
$$
\sX_\gamma = \Pi_0 + \sG_{\gamma,R}\circ \sX_\gamma + \sX_\gamma\circ \Delta_{\gamma,R}\,,
$$
$\bar \bP^{a,\gamma}$ almost surely, where $\Pi_0:E^\bN\to C(\bR_+,E)$, $\sG_{\gamma,R}:C(\bR_+,E)\to C(\bR_+,E)$ and $\Delta_{\gamma,R}:E^\bN\to E^\bN$ are the mappings respectively defined by
\begin{align*}
  &\Pi_0(x):t\mapsto x_0 \\
  &\Delta_{\gamma,R}(x):n\mapsto (x_n-x_{0}) - \gamma \sum_{k = 0}^{n-1} \int h_{\gamma,R}(s,x_k)\mu(ds) \\
  & \sG_{\gamma,R}(\sx): t\mapsto \int_0^t\int h_{\gamma,R}(s, \sx(\gamma \lfloor  u/\gamma\rfloor)) \, \mu(ds) du\ ,
\end{align*}
for every $x=(x_n:n\in \bN)$ and every $\sx\in C(\bR_+,E)$ .

\begin{lemma}
  \label{lem:ui}
For all $\gamma\in (0,\gamma_0]$ and all $x\in E^\bN$, define $Z^\gamma_{n+1}(x) \eqdef \gamma^{-1}(x_{n+1}-x_{n})$. 
There exists $\varepsilon>0$ such that:
\begin{align}
&\sup_{{n\in\bN, a\in \cK\cap \mD_{\gamma M} , \gamma\in (0,\gamma_0]}} \bar \bE^{a,\gamma}\left( \left(\|Z^\gamma_{n}\|+\frac{\bs d(x_{n})}{\gamma}\1_{\|x_{n}\|\leq R}\right)^{1+\varepsilon}\right)< +\infty \label{eq:ui} 
\end{align}
\end{lemma}

Using \cite[Lemma 4.2]{bia-hac-sal-(arxiv)16}, the uniform integrability
condition~(\ref{eq:ui}) implies\footnote{Lemma 4.2 of
\cite{bia-hac-sal-(arxiv)16} was actually shown with condition $[a\in \cK]$
instead of  $[a\in\cK\cap \mD_{\gamma M}]$, but the proof can be easily adapted
to the latter case.} 
that $\{\bar \bP^{a,\gamma}\sX_\gamma^{-1}: a\in  \cK\cap \mD_{\gamma M}, 
 \gamma\in (0,\gamma_0]\}$ is tight, and for any 
$T>0$,
\begin{equation}
\sup_{a\in \cK\cap \mD_{\gamma M}}\bar \bP^{a,\gamma}(\|\sX_\gamma\circ \Delta_{\gamma,R}\|_{\infty,T}>\varepsilon) \xrightarrow[]{\gamma\to 0}0\,,
\label{eq:rate-of-change-continuous}
\end{equation}
where the notation $\|\sx\|_{\infty,T}$ stands for the uniform norm of $\sx$ on $[0,T]$.

\begin{lemma}
  \label{lem:sko}
For an arbitrary sequence $(a_n,\gamma_n)$ such that 
$a_n\in \cK\cap \mD_{\gamma_n M}$ and $\gamma_n \to 0$, there 
exists a subsequence (still denoted as $(a_n,\gamma_n)$) such that 
$(a_n,\gamma_n)\to (a^*,0)$ for some $a^*\in \cK\cap \bmD$,
and there exists r.v. $\sz$ and $(\sx_n:n\in \bN)$ defined on some probability space $(\Omega',\mcF',\bP')$
into $C(\bR_+,E)$ s.t. $\sx_n$ has the distribution $\bar \bP^{a_n,\gamma_n}\sX_{\gamma_n}^{-1}$ and $\sx_n(\omega)\to \sz(\omega)$
for all $\omega\in \Omega'$.
Moreover, defining
$$
 u_n(t) \eqdef \sx_n(\gamma_n\lfloor t/\gamma_n\rfloor)\ ,
$$
the sequence $(a_n,\gamma_n)$ and $(\sx_n)$ can be chosen in such a way that 
the following holds $\bP'$-a.e.
\begin{equation}
\sup_n \int_0^T \left(\frac{\bs d(u_{n}(t))}{\gamma_n}\1_{\|u_{n}(t)\|\leq R}\right)^{1+\frac\varepsilon 2}dt\,<+\infty\quad (\forall T>0)\,,
\label{eq:int-d-bornee}
\end{equation}
where $\varepsilon>0$ is the constant introduced in Lem.~\ref{lem:ui}.
\end{lemma}
The limit $\sz$ satisfies the following:
\begin{lemma}
\label{lem:z-in-D}
Introduce the open ball $B_R\eqdef \{u\in E:\|u\|<R\}$.  The following holds $\bP'$-a.e.:
\begin{equation}
\forall t\geq 0,\ \sz(t)\in \bmD\cup B_R^c\,.\label{eq:z-in-D}
\end{equation}
\end{lemma}
Define \begin{equation*}
 v_n(s,t)\eqdef h_{\gamma_n,R}(s, u_n(t))\,.
\end{equation*}
Thanks to the convergence \eqref{eq:rate-of-change-continuous},
the following holds $\bP'$-a.e.: 
\begin{equation}
{\mathsf z}(t) = \sz(0)
 + \lim_{n\to\infty} \int_0^t\int_\Xi v_n(s,u)\, \mu(ds) \, du\,
 \qquad (\forall t\geq 0) \,.\label{eq:zintegral}
\end{equation}
We now select an $\omega\in \Omega'$ s.t. the events (\ref{eq:int-d-bornee}), (\ref{eq:z-in-D}) and (\ref{eq:zintegral}) are all realized, 
and omit the dependence in $\omega$ in the sequel.
Otherwise stated, $u_n$ and $v_n$ are handled from now on as determinitic functions, and no longer as random variables.
The aim of the next lemmas is to analyze the integrand $v_n(s,u)$.
Let $H_R(s,a)\eqdef -A(s,a)-B(s,a)$ if $\|a\|<R$, 
and $H_R(s,a)\eqdef E$ otherwise.
Denote the corresponding selection integral as $\sH_R(a) = \int H_R(s,a)\, \mu(ds)$. 
\begin{lemma}
  \label{lem:cv-HR}
For every $s$ $\mu$-a.e., it holds that for every $t\geq 0$,
$(u_n(t),v_n(s,t))\to \graph(H_R(s,\,.\,))$.
\end{lemma}
Consider some $T>0$ and let $\lambda_T$ represent the Lebesgue measure on the interval $[0,T]$.
To simplify notations, we set $\cL^{1+\varepsilon}_E\eqdef \cL^{1+\varepsilon}( \Xi\times [0,T], \mcG\otimes\mcB([0,T]), \mu\otimes\lambda_T; E)$.
\begin{lemma}
  \label{lem:v_bounded}
The sequence $(v_n:n\in \bN)$ forms a bounded subset of $\cL^{1+\varepsilon/2}_E$.
\end{lemma}
The sequence of mappings $((s,t)\mapsto (v_n(s,t),\|v_n(s,t)\|))$ is bounded in $\cL^{1+\varepsilon/2}_{E\times\bR}$
and therefore admits a weak cluster point in that space. We denote by $(v,w)$ such a cluster point,
where $v:\Xi\times  [0,T]\to E$ and $w:\Xi\times  [0,T]\to \bR$.
The following lemma is a consequence of Lem.~\ref{lem:cv-HR}.
\begin{lemma}
  \label{lem:cvth}
For every $(s,t)$ $\mu\otimes\lambda_T$-a.e.,
$(\sz(t),v(s,t))\in \graph(H_R(s,\,.\,))$.
\end{lemma}
By Lem.~\ref{lem:cvth} and Fubini's theorem, there is a $\lambda_T$-negligible set s.t.
for every $t$ outside this set, $v(\,.\,,t)$ is an integrable selection of $H_R(\,.\,,\sz(t))$.
Moreover, as $v$ is a weak cluster point of $v_n$ in $\cL^{1+\varepsilon/2}_{E}$, it holds that
$$
{\mathsf z}(t) = {\mathsf z}(0) 
 + \int_0^t\int_\Xi v(s,u)\, \mu(ds) \, du\,,
 \qquad (\forall t\in [0,T]) \,.
$$
Define $\sH_R(a)\eqdef \int H_R(\,.\,,a)d\mu$.  By the above equality,
$\sz$ is a solution to the DI $\dot{\sx}\in \sH_R(\sx)$ with initial
condition $\sz(0)=a^*$. Denoting by $\Phi_R(a^*)$ the set of such
solutions, this reads $\sz\in \Phi_R(a^*)$. As $a^*\in \cK\cap \bmD$,
one has $\sz\in \Phi_R(\cK\cap \bmD)$ where we use the notation
$\Phi_R(S) \eqdef \cup_{a\in S}\Phi_R(a)$ for every set $S\subset E$.
Extending the notation $\distC(x,\mathsf S) \eqdef \inf_{\sy\in
  \mathsf S}\distC(\sx,\sy)$, we obtain that
$\distC(\sx_n,\Phi_R(\cK\cap \bmD))\to 0$.
Thus, for every $\eta>0$, we have shown that 
$\bar \bP^{a_n,\gamma_n}(\distC(X_{\gamma_n},\Phi_R(\cK\cap \bmD))>\eta)\to 0$
as $n\to\infty$. We have thus proven the following result:
$$
\forall \eta>0,\ \lim_{\gamma\to 0}\sup_{a\in\cK\cap \mD_{\gamma M}} \bar\bP^{a,\gamma}(\distC(X_{\gamma},\Phi_R(\cK\cap \bmD))>\eta)=0\,.
$$
Letting $T>0$ and choosing $R>\sup\{\|\Phi(a,t)\|:t\in
[0,T], a\in \cK\cap \bmD\}$ (the latter quantity being finite, see
\emph{e.g.} \cite{bre-livre73}), it is easy to show that any solution
to the DI $\dot{\sx}\in \sH_R(\sx)$ with initial condition
$a\in\cK\cap \bmD$ coincides with $\Phi(a,\,.\,)$ on $[0,T]$. 
By the same arguments as in  \cite[Section 4 - End
of the proof]{bia-hac-sal-(arxiv)16}, Theorem~\ref{fb-apt} follows. 

\section{Cluster points of the $P_\gamma$ invariant measures. 
End of the proof of Th.~\ref{the:CV}}
\label{sec-clust} 

\begin{lemma} 
\label{fb-PH} 
Assume that there exists $x_\star \in Z(\mA + \mB)$ that admits a 
${\mathcal L}^{2}$ representation. Then,  
\[
P_\gamma(x,  \|\cdot-x_\star\|^2) \leq \| x - x_\star\|^2 - 
  0.5 \gamma \psi_\gamma(x) + \gamma^2 C , 
\]
where $\psi_\gamma$ is the function defined in~\eqref{psig}. 
\end{lemma} 
\begin{proof}
By assumption, there exists a ${\mathcal L}^{2}$ representation $(\varphi, B)$
of $x_\star$. By expanding 
\[
\|x_{n+1}-x_\star\|^2 = 
\|x_{n}-x_\star\|^2+2\ps{x_{n+1}-x_n,x_n-x_\star} + \|x_{n+1}-x_n\|^2\, , 
\]
and by using \eqref{fb-yosida}, we obtain
\begin{multline}
\label{x-x*} 
\|x_{n+1}-x_\star\|^2    
= \|x_{n}-x_\star\|^2 
 - 2\gamma\ps{A_{\gamma}(\xi_{n+1},x_n-\gamma B(\xi_{n+1},x_n)) 
  + B(\xi_{n+1},x_n),x_n-x_\star} \\
  + \gamma^2\| A_{\gamma}(\xi_{n+1},x_n-\gamma B(\xi_{n+1},x_n)) 
   + B(\xi_{n+1},x_n)\|^2 . 
\end{multline}
Write $x = x_n$, 
$A_\gamma = A_{\gamma}(\xi_{n+1},x_n-\gamma B(\xi_{n+1},x_n))$, 
$J_\gamma = J_{\gamma}(\xi_{n+1},x_n-\gamma B(\xi_{n+1},x_n))$,  
$B = B(\xi_{n+1},x_n)$, $B_\star = (\xi_{n+1},x_\star)$, and 
$\varphi = \varphi(\xi_{n+1})$ for conciseness. 
We write 
\begin{align*} 
\ps{A_\gamma, x - x_\star} &= \ps{A_\gamma - \varphi, J_\gamma - x_\star}
+ \ps{A_\gamma - \varphi, x - \gamma B - J_\gamma} + 
\gamma \ps{A_\gamma - \varphi,B} \\
&\phantom{=} + \ps{\varphi, x - x_\star}  \\
&= \ps{A_\gamma - \varphi, J_\gamma - x_\star} + 
\gamma \| A_\gamma\|^2 - \gamma \ps{A_\gamma,\varphi} 
+ \gamma \ps{A_\gamma-\varphi, B} + 
\ps{\varphi, x - x_\star} .  
\end{align*}
We also write 
$\ps{B, x - x_\star} = \ps{B - B_\star, x - x_\star} + 
\ps{B_\star, x - x_\star}$ and 
$\gamma^2\| A_{\gamma} + B \|^2 = \gamma^2 (\| A_\gamma\|^2 + \| B\|^2 + 
2\ps{A_\gamma,B} )$. Plugging these identities at the right hand side 
of~\eqref{x-x*}, we obtain 
\begin{align*}
\|x_{n+1}-x_\star\|^2 &= \|x - x_\star\|^2 
-2\gamma \left\{ \ps{A_\gamma-\varphi, J_\gamma - x_\star} 
  + \ps{B - B_\star, x - x_\star}  \right\} 
 -  \gamma^2 \| A_\gamma\|^2 \\
&\phantom{=} + 2\gamma^2 \ps{A_\gamma,\varphi} 
 +2 \gamma^2 \ps{\varphi, B} + \gamma^2 \| B\|^2 
-2\gamma \ps{\varphi+B_\star, x - x_\star} \\
&\leq \|x - x_\star\|^2 
-2\gamma \left\{ \ps{A_\gamma-\varphi, J_\gamma - x_\star} 
  + \ps{B - B_\star, x - x_\star}  \right\} 
 -  (\gamma^2/2) \| A_\gamma\|^2 \\
&\phantom{=} + (3\gamma^2/2) \| B\|^2 + 4 \gamma^2 \| \varphi \|^2 
-2\gamma \ps{\varphi+B_\star, x - x_\star} \\ 
&\leq \|x - x_\star\|^2 
-2\gamma \left\{ \ps{A_\gamma-\varphi, J_\gamma - x_\star} 
  + \ps{B - B_\star, x - x_\star}  \right\} 
 -  (\gamma^2/2) \| A_\gamma\|^2 \\
&\phantom{=} + 3\gamma^2 \| B - B_\star \|^2 + 3 \gamma^2 \|B_\star\|^2 
 + 4 \gamma^2 \| \varphi \|^2 
-2\gamma \ps{\varphi+B_\star, x - x_\star} 
\end{align*} 
where the first inequality is due to the fact that $2\ps{a,b} \leq \| a\|^2/2 +
2 \|b\|^2$ and the second to the triangle inequality. 
Observe that the term between the braces at the right hand side of the last 
inequality is nonnegative thanks to the monotonicity of $A(s,\cdot)$ and
$B(s,\cdot)$. Taking the conditional expectation $\bE_n$ at each side, the 
contribution of the last inner product at the right hand side disappears, and 
we obtain 
\[
P_\gamma(x, \|\cdot-x_\star\|^2) \leq \| x - x_\star\|^2 - 
  0.5 \gamma \psi_\gamma(x) + 4 \gamma^2 \int \| \varphi(s) \|^2 \mu(ds) 
+ 3 \gamma^2 \int \|B(s,x_\star) \|^2 \mu(ds) 
\]
where $\psi_\gamma$ is the function defined in~\eqref{psig}. 
\end{proof} 

Given $k\in\bN$, we denote by $P_\gamma^k$ the kernel $P_\gamma$ iterated
$k$ times. The iterated kernel is defined recursively as 
$P_\gamma^0(x,dy) = \delta_x(dy)$, and 
\[
P_\gamma^k(x, S) = \int P_\gamma^{k-1} (y, S) \, P_\gamma(x, dy) 
\]
for each $S \in \mcB(E)$. 

\begin{lemma} 
\label{inv-mon}
Let the assumptions of the statement of Th.~\ref{fb-apt} hold true. 
Assume that for all $\varepsilon > 0$, there exists $M > 0$ such that 
\begin{equation}
\label{supp-I(P)} 
\sup_{\gamma \in (0,\gamma_0]} \sup_{\pi\in \cI(P_{\gamma})} \pi((\mD_{M\gamma})^c) \leq \varepsilon .
\end{equation} 
Then, as $\gamma\to 0$, any cluster point of $\cI(\cP)$ is an element of 
$\cI(\Phi)$.
\end{lemma} 
Note that in the common domain case, \eqref{supp-I(P)} is trivially satisfied,
since the supports of all the invariant measures are included in $\bmD$. 

\begin{proof}
Choose two sequences $(\gamma_i)$ and $(\pi_i)$ such that $\gamma_i\to 0$, 
$\pi_i\in \cI(P_{\gamma_i})$ for all $i\in\bN$, and $\pi_i$ converges 
narrowly to some $\pi \in\cM( E)$ as $i\to\infty$. 

Let $f$ be a real, bounded, and Lipschitz function on $E$ with Lipschitz
coefficient $L$. By definition, $\pi_i(f)=\pi_i(P^k_{\gamma_i}f)$
for all $k\in\bN$. Set $t>0$, and let $k_i=\lfloor t/\gamma_i \rfloor$. 
We have 
\begin{align*}
|\pi_i f - \pi_i (f\circ\Phi(\Pi_{\bmD}(\cdot), t))| 
&= \left| \int (P_{\gamma_i}^{k_i}(a,f)-f(\Phi(\Pi_{\bmD}(a),t))) 
                                            \pi_i(da)\right|  \\ 
&\leq \int \left| P_{\gamma_i}^{k_i}(a,f)
   -f(\Phi(\Pi_{\bmD}(a),{k_i}\gamma_i)) \right| \pi_i(da)  \\
&\phantom{=} +  \int \left| f(\Phi(\Pi_{\bmD}(a),{k_i}\gamma_i)) 
 - f(\Phi(\Pi_{\bmD}(a),t))\right| \pi_i(da)   \\ 
&\leq  \int \bE^{a,\gamma_i} \left|f(x_{k_i})
      -f(\Phi(\Pi_{\bmD}(a),{k_i}\gamma_i))\right| \pi_i(da) \\
&\phantom{=} 
+ \int \left| f(\Phi(\Pi_{\bmD}(a),{k_i}\gamma_i)) 
            - f(\Phi(\Pi_{\bmD}(a),t))\right|  \pi_i(da)   \\ 
&\eqdef U_i + V_i \, . 
\end{align*}
By the boundedness and the Lispchitz-continuity of $f$, 
\[
U_i \leq \int  \bE^{a,\gamma_i} \left[2\|f\|_\infty \wedge 
  L \| x_{k_i} - \Phi(\Pi_{\bmD}(a),k_i\gamma_i) \| \right] \pi_i(da)\,.
\]
Fixing an arbitrarily small $\varepsilon > 0$, it holds by \eqref{supp-I(P)} 
that 
$\pi_i((\mD_{M\gamma_i})^c) \leq \varepsilon / 2$ for a large enough $M$. 
By the tightness of $(\pi_i)$, we can choose a compact $\cK\subset E$ s.t.~for 
all $i$, $\pi_i(\cK^c) \leq \varepsilon/2$. With these choices, we obtain 
\[
U_i \leq 
\sup_{a\in\cK\cap \mD_{M\gamma_i}} 
 \bE^{a,\gamma_i} \left[2\|f\|_\infty \wedge 
  L \| x_{k_i} - \Phi(\Pi_{\bmD}(a),k_i\gamma_i) \| \right] 
 + 2 \|f\|_\infty \,\varepsilon\,.
\]
Denoting as $(\cdot)_{[0,t]}$ the restriction of a function to the interval 
$[0,t]$, and observing that 
$\| x_{k_i} - \Phi(\Pi_{\bmD}(a),k_i\gamma_i) \| \leq 
  \| (\sX_\gamma(x) - \Phi(\Pi_{\bmD}(a),\cdot))_{[0,t]} \|_\infty$, we can 
now  apply Th.~\ref{fb-apt} to obtain 
\[
\sup_{a\in\cK\cap \mD_{M\gamma_i}} 
 \bE^{a,\gamma_i} \left[2\|f\|_\infty \wedge 
  L \| x_{k_i} - \Phi(\Pi_{\bmD}(a),k_i\gamma_i) \| \right] 
\xrightarrow[i \to\infty]{} 0\, . 
\]
As $\varepsilon$ is arbitrary, we obtain that $U_i \to_i 0$. Turning to $V_i$, 
fix an arbitrary $\varepsilon > 0$, and choose a compact $\cK \subset E$ such
that $\pi_i(\cK^c) \leq \varepsilon$ for all $i$. We have 
\[
V_i \leq \sup_{a\in\cK} \left| f(\Phi(\Pi_{\bmD}(a),{k_i}\gamma_i)) 
            - f(\Phi(\Pi_{\bmD}(a),t))\right|  + 
2 \| f\|_\infty \varepsilon \, .
\]
By the uniform continuity of the function 
$f \circ \Phi(\Pi_{\bmD}(\cdot), \cdot)$ on the compact $\cK \times [0,t]$, and
by the convergence $k_i\gamma_i \uparrow t$, we obtain that 
$\limsup_i V_i \leq 2 \| f\|_\infty \varepsilon$. As $\varepsilon$ is 
arbitrary, $V_i \to_i 0$.  
In conclusion, $\pi_i f - \pi_i (f\circ\Phi(\Pi_{\bmD}(\cdot), t)) \to_i 0$. 
Moreover, 
$\pi_i f - \pi_i (f\circ\Phi(\Pi_{\bmD}(\cdot), t)) \to_i 
\pi f - \pi (f\circ\Phi(\Pi_{\bmD}(\cdot), t))$ since 
$f(\cdot) - f\circ\Phi(\Pi_{\bmD}(\cdot), t))$ is bounded continuous. Thus, 
$\pi f = \pi (f\circ\Phi(\Pi_{\bmD}(\cdot), t))$. 
Since $\pi_i$ converges narrowly to $\pi$, we obtain that for all $\eta > 0$, 
$\pi(\cl(\mD_\eta)^c) \leq\liminf_i \pi_i(\cl(\mD_\eta)^c) = 0$ by choosing
$\varepsilon$ arbitrarily small in \eqref{supp-I(P)} and making 
$\gamma_i\to 0$. Thus, $\support(\pi) \subset \bmD$, and we obtain in conclusion
that $\pi f = \pi(f\circ\Phi(\cdot, t))$ for an arbitrary real, bounded, 
and Lipschitz continuous function $f$. Thus, $\pi\in\cI(\Phi)$. 
\end{proof}

To establish \eqref{supp-I(P)} in the different domains case, we need the
following lemma. 
\begin{lemma}
\label{thickening} 
Let Assumptions \ref{linreg}, \ref{JBgrow}, and 
\ref{psi-increase}--\ref{psi-sous-lin} hold true. Then, for all 
$\varepsilon > 0$, there exists $M > 0$ such that 
\[
\sup_{\gamma \in (0,\gamma_0]} \sup_{\pi\in \cI(P_{\gamma})} \pi((\mD_{M\gamma})^c) \leq \varepsilon .
\]
\end{lemma}
\begin{proof}
We start by writing 
\begin{align*} 
\bs d(x_{n+1}) 
  \leq \| x_{n+1} - \Pi_{\cl(\mD)}(x_{n}) \|   
\leq \| x_{n+1} - \Pi_{\cl(D(\xi_{n+1}))}(x_n) \| + 
  \|\Pi_{\cl(D(\xi_{n+1}))}(x_n) - \Pi_{\cl(\mD)}(x_n) \| . 
\end{align*} 
On the one hand, we have by Assumption~\ref{JBgrow} and the nonexpansiveness
of the resolvent that 
\begin{align*} 
\bar \bE_n^{a,\gamma} \| x_{n+1} - \Pi_{\cl(D(\xi_{n+1}))}(x_n) \| &\leq 
 \bar \bE_n^{a,\gamma} \| J_\gamma(\xi_{n+1}, x_n) - \Pi_{\cl(D(\xi_{n+1}))}(x_n) \| 
 + \gamma \bar \bE_n^{a,\gamma} \| B(\xi_{n+1}, x_n) \| \\ 
&\leq C \gamma ( 1 + \Psi(x_n) ) \, , 
\end{align*} 
on the other hand, since 
\[
\| \Pi_{\cl(D(\xi_{n+1}))}(x_n) - \Pi_{\cl(\mD)}(x_n) \|^2 \leq 
\bs d(x_n)^2 - d(x_n, D(\xi_{n+1}))^2  
\quad \text{(see \eqref{pi-fne})}, 
\]
we can make use of Assumption \ref{linreg} to obtain 
\[
\bar \bE_n^{a,\gamma} \| \Pi_{\cl(D(\xi_{n+1}))}(x_n) - \Pi_{\cl(\mD)}(x_n) \| 
 \leq 
(\bar \bE_n^{a,\gamma} \| \Pi_{\cl(D(\xi_{n+1}))}(x_n) 
  - \Pi_{\cl(\mD)}(x_n) \|^2)^{1/2} 
\leq 
\rho \bs d(x_n) \, , 
\]
where $\rho \in [0, 1)$. We therefore obtain that 
$\bar \bE_n^{a,\gamma} \bs d(x_{n+1}) \leq \rho \bs d(x_n) + C \gamma (1 + \Psi(x_n))$. 
By iterating, we end up with the inequality 
\begin{equation}
\label{Pineq} 
P_\gamma^{n+1}(a, \bs d) \leq \rho^{n+1} \bs d(a) + C \gamma 
\sum_{k=0}^n \rho^{n-k} (1 + P_\gamma^k(a, \Psi) ) . 
\end{equation} 
By Lem.~\ref{fb-PH}, $\psi_\gamma(x) \leq 2 \gamma^{-1} \| x - x_\star \|^2 
+ \gamma C$, thus $P_\gamma(a, \psi_\gamma) \leq 
2 \gamma^{-1} P_\gamma( a, \|\cdot - x_\star \|^2) + \gamma C 
\leq 2 \gamma^{-1} \| a - x_\star \|^2 + C < \infty$. We obtain 
similarly that $P_\gamma^k(a, \psi_\gamma) < \infty$, thus 
$P_\gamma^k(a, \Psi) < \infty$ for all $k\in \bN$. 

Since $(\mD_{M\gamma})^c = \{ x \, : \, \bs d(x) \geq M \gamma \}$, it holds by 
Markov's inequality that 
\[
P_\gamma^k(a, (\mD_{M\gamma})^c) \leq \frac{P_\gamma^k(a, \bs d)}{M\gamma}  
\]
for all $k\in \bN$. Let $\pi_\gamma$ be a $P_\gamma$-invariant probability 
measure. From Assumption~\ref{psi-increase}--\ref{psi-sous-lin} and 
Lem.~\ref{fb-PH}, the inequality \eqref{eq:PH} in the statement
of Prop.~\ref{prop:PH} is satisfied with 
$V(x) = \| x - x_\star\|^2$,  $Q(x) = \Psi(x)$, $\alpha(\gamma) = \gamma/2$, 
and $\beta(\gamma) = C \gamma^2$. By the first part of this proposition, 
$\sup_\gamma \pi_\gamma \Psi < \infty$. In particular, noting that
$\bs d(x) \leq \| x\| + \| \Pi_{\cl(\mD)}(0) \|$, we obtain that 
$\sup_\gamma \pi_\gamma \bs d < \infty$. Getting back to~\eqref{Pineq}, we 
have for all $n\in\bN$, 
\begin{align*}
\pi_\gamma((\mD_{M\gamma})^c) &= 
P_\gamma^{n+1}(\pi_\gamma, (\mD_{M\gamma})^c) \\
&\leq  \frac{P_\gamma^{n+1}(\pi_\gamma, \bs d )}{M\gamma} \\
&\leq  \rho^{n+1} \frac{\pi_\gamma \bs d}{M\gamma} 
+ \frac{C}{M} 
\sum_{k=0}^n \rho^{n-k} (1 + P_\gamma^k(\pi_\gamma, \Psi) ) \\
&= 
\rho^{n+1} \frac{\pi_\gamma \bs d}{M\gamma} + \frac{C}{M} 
\sum_{k=0}^n \rho^{n-k} (1 + \pi_\gamma \Psi ) \\
&\leq \rho^{n+1} \frac{C}{M\gamma} + \frac{C}{M} \, . 
\end{align*} 
By making $n\to\infty$, we obtain that 
$\pi_\gamma((\mD_{M\gamma})^c) \leq C / M$, and the proof is concluded by 
taking $M$ as large as required.  
\end{proof} 

\subsection*{Th.~\ref{the:CV}: proofs of the convergences \eqref{cvg:support}, 
\eqref{P:CVS}, and \eqref{E:CVS}} 

We need to check that the assumptions of Prop.~\ref{prop:PH} are satisfied. 
Lem.~\ref{fb-PH} shows that the inequality \eqref{eq:PH} is satisfied with 
$V(x) = \| x - x_\star\|^2$,  $Q(x) = \Psi(x)$, $\alpha(\gamma) = \gamma/2$, 
and $\beta(\gamma) = C \gamma^2$, and  
Assumption~\ref{psi-increase}--\ref{psi-sous-lin} ensures that 
$\Psi(x) \xrightarrow[\|x\|\to\infty]{} \infty$ as required. 

When the assumptions of Th.~\ref{fb-apt}, are satisfied, Lem.~\ref{inv-mon}
shows with the help of Lem.~\ref{thickening} when needed that any cluster point
of $\cI(\cP)$ belongs to $\cI(\Phi)$. The required convergences follow at
once from Prop.~\ref{prop:PH}. Theorem~\ref{the:CV} is proven.




\appendix 


\section{Proofs relative to Section~\ref{sec-tightness}}
\label{anx:case-studies}

\subsection{Proof of Prop. \ref{tight-stogra}}

It is well known that the coercivity or the supercoercivity of a function 
$q \in \Gamma_0$ can be characterized through the study of the recession
function $q^\infty$ of $q$, which is the function in $\Gamma_0$ whose epigraph
is the recession cone of the epigraph of $q$ \cite[\S 8]{Roc70},
\cite[\S~6.8]{lau-(livre)72}.  
We recall the following fact.
\begin{lemma}
\label{coer-rec} 
The function $q \in \Gamma_0$ is coercive if and only if $0$ is the only 
solution of the inequality $q^\infty(x) \leq 0$. It is supercoercive if and 
only if $q^\infty = \iota_{\{0\}}$. 
\end{lemma} 
\begin{proof} 
By \cite[Prop.~6.8.4]{lau-(livre)72}, 
$\lev_{\leq 0} q^\infty$ is the recession cone of any level set $\lev_{\leq a} q$
which is not empty \cite[Th.~8.6]{Roc70}.
Thus, $q$ is coercive if and only if $\lev_{\leq 0} q^\infty$ is the recession cone
of a nonempty compact set, hence equal to $\{0\}$.
The second point follows from \cite[Prop. 2.16]{bauschke1997legendre}.
\end{proof} 

\begin{lemma}
\label{rec-moreau} 
For each $\gamma > 0$, $q^\infty = (q_\gamma)^\infty$. 
\end{lemma}
\begin{proof}
By \cite[Th.~6.8.5]{lau-(livre)72}, the Legendre-Fenchel transform $(q^\infty)^*$ of $q^\infty$
satisfies $(q^\infty)^* = \iota_{\cl{\dom q^*}}$.
Since $q_\gamma = q \, \square \, ((2\gamma)^{-1} \| \cdot \|^2 )$ where 
$\square$ is the infimal convolution operator, $(q_\gamma)^* = q^* + 
(\gamma/2) \| \cdot \|^2$. Therefore, $\dom q^* = \dom (q_\gamma)^*$, which 
implies that $(q^\infty)^* = ((q_\gamma)^\infty)^*$, and the result follows. 
\end{proof} 

\begin{lemma}[{\cite[Th.~II.2.1]{hir-76}}] 
\label{int<->rec}
Assume that 
$q : \Xi \times  E \to (-\infty, \infty]$ is a normal integrand
such that $q(s, \cdot) \in \Gamma_0$ for almost every $s$.
Assume that $Q(x) \eqdef \int q(s,x) \, \mu(ds)$ belongs to $\Gamma_0$. Then, 
$Q^\infty(x) = \int q^{\infty}(s,x) \, \mu(ds)$, where $q^\infty(s,\cdot)$ is
the recession function of $q(s,\cdot)$. 
\end{lemma} 

We now enter the proof of Prop.~\ref{tight-stogra}. 
Denote by $g_\gamma(s,\cdot)$ the Moreau envelope of the mapping $g(s,\cdot)$
defined above.
\begin{lemma}
\label{Ggamma} 
Let Hypothesis \ref{L2} hold true. Then, for all $\gamma > 0$, the mapping
\[
G^\gamma:x \mapsto \int g_\gamma(s,x) \, \mu(ds) \, ,  
\]
is well defined on $E\to \bR$, and is convex (hence continuous) on $E$. 
Moreover, $G^\gamma \uparrow G$ as $\gamma\downarrow 0$. 
\end{lemma} 
\begin{proof}
Since $x_\star \in \dom G$ from Hypothesis \ref{L2}, it holds from the 
definition of the function $g$ that 
$\int | g(s, x_\star) | \, \mu(ds) < \infty$. Moreover, 
noting that $\varphi(s) \in \partial g(s,x_\star)$, the inequality  
$g(s,x) \geq \ps{\varphi(s), x-x_\star} + g(s,x_\star)$ holds. Thus, 
\begin{align*}
g_\gamma(s,x) &= 
 \inf_w \Bigl( g(s,w) + \frac{1}{2\gamma} \| w - x \|^2 \Bigr)
\geq \inf_w \Bigl( \ps{\varphi(s), w-x_\star} + g(s,x_\star)
+ \frac{1}{2\gamma} \| w - x \|^2 \Bigr) \\
&= \ps{\varphi(s), x-x_\star} + g(s,x_\star) 
  - \frac{\gamma}{2} \| \varphi(s)  \|^2 . 
\end{align*} 
Writing $x = x^+ - x^-$ where $x^+ = x \vee 0$, this inequality shows that 
$g_\gamma(\cdot,x_\star)^-$ is integrable. Moreover, since the Moreau envelope 
satisfies $g_\gamma(s,x) \leq g(s,x)$, we obtain that 
$g_\gamma(\cdot,x_\star)^+ \leq g(\cdot,x_\star)^+ \leq | g(\cdot,x_\star) |$ 
who is also integrable. Therefore, $| g_\gamma(\cdot,x_\star) |$ is integrable. 
For other values of $x$, we have 
\[
g_\gamma(s,x) = g_\gamma(s,x_\star) + \int_0^1 
\ps{ x - x_\star, \nabla g_\gamma( s, x_\star + t(x-x_\star)) - 
\nabla g_\gamma( s, x_\star) } \, dt  + 
\ps{ x - x_\star, \nabla g_\gamma( s, x_\star)}, 
\]
where $\nabla g_\gamma(s,x)$ is the gradient of $g_\gamma(s,x)$ w.r.t.~$x$. 
Using the well know properties of the Yosida regularization 
(see Sec.~\ref{basic}), we obtain 
\[
| g_\gamma(s,x) | \leq | g_\gamma(s,x_\star) | + 
\frac{\| x - x_\star\|^2}{2\gamma} 
 + \| x - x_\star\| \, \| \varphi(s) \|^2 . 
\]
Consequently, $g_\gamma(\cdot, x)$ is integrable, thus, $G^\gamma(x)$ is 
defined for all $x\in E$. The convexity and hence the continuity of 
$G^\gamma$ follow trivially from the convexity of $g_\gamma(s,\cdot)$. 

Since the integrand $g_\gamma(s,x)$ increases as $\gamma$ decreases, so is
the case of $G^\gamma(x)$. If $x \in \dom(G)$, it holds that 
$| g(\cdot, x) |$ is integrable. On the one hand, 
$g_\gamma(s,x)^+ \leq | g(s,x) |$, and on the other hand, 
$g_\gamma(s,x)^- \leq 
\| \varphi(s) \| \| x-x_\star\| + | g(s,x_\star) | +  \| \varphi(s)  \|^2$
for $\gamma\leq 2$. By the dominated convergence, $G^\gamma(x) \to G(x)$ as
$\gamma\to 0$. 
If $x \not\in \dom G$, then $\int g_\gamma(s,x)^+ \mu(ds) \to\infty$ as 
$\gamma\to 0$ by monotone convergence, and $\int g_\gamma(s,x)^- \mu(ds)$ remains 
bound. Thus, $G^\gamma(x) \to\infty$. 
\end{proof} 

\begin{lemma}
\label{bnd-psi} 
Let Hypotheses \ref{L2} and \ref{baillon} hold true. Then, for all $\gamma$ 
small enough, 
\[
G^\gamma(x) + F(x) - G^\gamma(x_\star) - F(x_\star) \leq 
2 \psi_\gamma(x) + \gamma C , 
\]
where $\psi_\gamma$ is given by~\eqref{psig}. 
\end{lemma} 
\begin{proof} 
By the convexity of $g_\gamma(s,\cdot)$ and $f(s,\cdot)$, we have 
\begin{gather*} 
g_\gamma(s, x - \gamma\nabla f(s,x)) - g_\gamma(s, x_\star) 
\leq \ps{\nabla g_\gamma(s, x - \gamma \nabla f(s,x)), 
  x - \gamma\nabla f(s,x) - x_\star} , \ \text{and} \\ 
f(s, x) - f(s, x_\star) - \ps{\nabla f(s,x_\star), x-x_\star} \leq 
\ps{\nabla f(s,x) - \nabla f(s,x_\star), x - x_\star} .  
\end{gather*} 
Write $g_\gamma = g_\gamma(s, x - \gamma\nabla f(s,x))$, 
$\nabla f = \nabla f(s,x)$,  
$\prox_\gamma = \prox_{\gamma g(s,\cdot)}(x - \gamma \nabla f(s,x))$, 
$\varphi = \varphi(s)$, and $\nabla f_\star = \nabla f(s,x_\star)$. 
From these two inequalities, we obtain 
\begin{align*}
&g_\gamma(s, x - \gamma\nabla f(s,x)) - g_\gamma(s,x_\star) + f(s,x) 
- f(s,x_\star) - \ps{\varphi(s)+\nabla f(s,x_\star), x-x_\star}  \\
&\leq \ps{\nabla g_\gamma, x - \gamma\nabla f - x_\star + 
           \prox_\gamma - \prox_\gamma} 
  + \ps{\nabla f - \nabla f_\star, x - x_\star} 
  - \ps{\varphi, x - x_\star + \prox_\gamma - \prox_\gamma}
   \\
&= \ps{\nabla g_\gamma - \varphi, \prox_\gamma - x_\star} 
  + \ps{\nabla f - \nabla f_\star, x - x_\star} 
  + \gamma \| \nabla g_\gamma \|^2 
 - \gamma\ps{\varphi, \nabla g_\gamma +\nabla f} . 
\end{align*}
Again, by the convexity of $g_\gamma(s,\cdot)$, we have 
\[
g_\gamma(s,  x - \gamma\nabla f(s,x)) \geq g_\gamma(s,x) - 
\gamma\ps{\nabla g_\gamma(s,x), \nabla f(s,x)} .
\]
Thus, we obtain 
\begin{align*}
&g_\gamma(s, x) - g_\gamma(s,x_\star) + f(s,x) - f(s,x_\star) 
 - \ps{\varphi(s)+\nabla f(s,x_\star), x-x_\star} \\
&\leq \ps{\nabla g_\gamma - \varphi, \prox_\gamma - x_\star} 
  + \ps{\nabla f - \nabla f_\star, x - x_\star} 
  + \gamma \| \nabla g_\gamma \|^2 
 - \gamma\ps{\varphi, \nabla g_\gamma + \nabla f}
 + \gamma\ps{\nabla g_\gamma(s,x), \nabla f} .
\end{align*} 
We now bound the sum of the last two terms at the right hand side. 
By the $\gamma^{-1}$-Lipschitz continuity of the Yosida regularization, 
$|\ps{\nabla g_\gamma(s,x) - \nabla g_\gamma, \nabla f}| \leq 
\|\nabla f\|^2$. Using in addition the inequalities $|\ps{a,b}| \leq 
\| a \|^2 / 2 + \| b \|^2 / 2$ and  
$\|  \nabla f \|^2 \leq 2 \| \nabla f_\star \|^2 + 
2 \| \nabla f - \nabla f_\star \|^2$, we obtain 
\begin{align*} 
\gamma\ps{\nabla g_\gamma(s,x), \nabla f}  
- \gamma\ps{\varphi, \nabla g_\gamma + \nabla f} &= 
\gamma\ps{\nabla g_\gamma(s,x) - \nabla g_\gamma, \nabla f}  
+ \gamma\ps{\nabla g_\gamma, \nabla f} 
- \gamma\ps{\varphi, \nabla g_\gamma + \nabla f} \\
&\leq 2\gamma \| \nabla f \|^2 + \gamma \| \nabla g_\gamma \|^2 + 
\gamma \| \varphi \|^2  \\
&\leq 4\gamma \| \nabla f -\nabla f_\star \|^2 
 + 4\gamma \| \nabla f_\star \|^2 + \gamma \| \nabla g_\gamma \|^2 + 
\gamma \| \varphi \|^2 . 
\end{align*} 
Thus, 
\begin{align*}
&g_\gamma(s, x) - g_\gamma(s,x_\star) + f(s,x) - f(s,x_\star) 
 - \ps{\varphi(s)+\nabla f(s,x_\star), x-x_\star} \\
&\leq 2 \left( \ps{\nabla g_\gamma - \varphi, \prox_\gamma - x_\star} 
  + \ps{\nabla f - \nabla f_\star, x - x_\star} 
  + \gamma \| \nabla g_\gamma \|^2 
  - 6 \gamma \| \nabla f - \nabla f_\star \|^2 \right) \\
&\phantom{\leq} + 16 \gamma \| \nabla f - \nabla f_\star \|^2 - 
   \ps{\nabla f - \nabla f_\star, x - x_\star}  
 + \gamma \| \varphi \|^2 + 4 \gamma \|  \nabla f_\star \|^2 \, . 
\end{align*} 
Taking the integral with respect to $\mu(ds)$ at both sides, the contribution
of the inner product $\ps{\varphi+\nabla f_\star, x-x_\star}$ vanishes. 
Recalling~\eqref{psig}, we obtain 
\begin{align*}
& G_\gamma(x) + F(x) - G_\gamma(x_\star) - F(x_\star) \\
&\leq 2 \psi_\gamma(x) - 
\int ( \ps{\nabla f - \nabla f_\star, x - x_\star} - 
     16 \gamma \| \nabla f - \nabla f_\star \|^2 ) \, d\mu 
+ \gamma \int (\| \varphi \|^2 + 4 \|  \nabla f_\star \|^2) \, d\mu \, . 
\end{align*} 
Using Hypothesis~\ref{baillon}, we obtain the desired result. 
\end{proof} 

\paragraph*{End of the proof of Prop.~\ref{tight-stogra}.} 
Let $\gamma_0 > 0$ be such that Lem.~\ref{bnd-psi} holds true for all 
$\gamma\in (0,\gamma_0]$. Denoting as $q(s,\cdot)^\infty$ the recession 
function of $q(s,\cdot)$, we have 
\[
(G^{\gamma_0} + F)^\infty 
\stackrel{\text{(a)}}{=} 
\int \left( (g_{\gamma_0}(s,\cdot))^\infty + f(s,\cdot)^\infty \right) 
  \, \mu(ds) 
\stackrel{\text{(b)}}{=} 
\int \left( g(s,\cdot)^\infty + f(s,\cdot)^\infty \right) \, \mu(ds) 
\stackrel{\text{(c)}}{=} 
(G + F)^\infty \, , 
\]
where the equalities (a) and (c) are due to Lem.~\ref{int<->rec}, and 
(b) is due to Lem.~\ref{rec-moreau}. Thus, by Lem.~\ref{coer-rec}, 
$F+G$ is coercive (resp.~supercoercive) if and only if $F+G^{\gamma_0}$ is
coercive (resp.~supercoercive). Consequently, since $G^\gamma$ increases as 
$\gamma$ decreases by Lem.~\ref{Ggamma}, 
the hypotheses \ref{L2}, \ref{baillon}, and \ref{F+G-coerc}--\ref{Zer-cpct} 
(resp., \ref{L2}, \ref{baillon}, \ref{F+G-coerc}--\ref{F+G-super}) 
imply Assumption~\ref{psi-increase}--\ref{psi-sous-lin} 
(resp.~Assumption~\ref{psi-increase}--\ref{psi-linear}). 
Prop.~\ref{tight-stogra} is proven. 

\subsection{Proof of Lem.~\ref{lm:L2}} 

We first recall that $\partial G(\cdot) = \int \partial g(s,\cdot) \, \mu(ds)$,
where  
\[
\partial g(s, \cdot) = 
 \left\{\begin{array}{ll} \alpha(0)^{-1} \partial h(u,\cdot) &
\text{if } i = 0, \\
\partial \iota_{\mC_i} & \text{otherwise}, 
\end{array}\right. 
\]
for $s = (u,i) \in \Xi$. Let $\psi$ be an arbitrary measurable 
$\Sigma \to E$ function such that $\psi(u) \in \partial h(u, x_\star)$ for 
$\zeta$-almost all $u\in\Sigma$ (such functions are called \emph{measurable 
selections} of the set-valued function $\partial h(\cdot, x_\star)$). For each 
$d\in E$, it holds by the convexity of $h(u,\cdot)$ that 
\begin{align*} 
h(u, x_\star + d) &\geq h(u, x_\star) + \ps{\psi(u), d}, \ \text{and} \\
h(u, x_\star - d) &\geq h(u, x_\star) - \ps{\psi(u), d}, 
\end{align*} 
for $\zeta$-almost all $u\in\Sigma$. Equivalently, 
\[
h(u, x_\star) - h(u, x_\star - d) \leq 
\ps{\psi(u), d} \leq h(u, x_\star + d) - h(u, x_\star) . 
\] 
Thus, if $\| d \|$ is small enough but otherwise $d$ is arbitrary, we get
from the second assumption of the statement that $\ps{\psi(u), d}$ is 
$\zeta$-square-integrable. Thus, $\int \|\psi(u)\|^2 \, \zeta(du) < \infty$ 
(see \cite[Th.~II.4.2]{hir-76} for a similar argument). 
Now, writing $s = (u,i) \in \Xi$, every measurable selection $\phi$ of 
$\partial g(\cdot, x_\star)$ is of the form 
\[
\phi(s) = 
 \left\{\begin{array}{ll} \alpha(0)^{-1} \psi(u) & \text{if } i = 0, \\
  \theta_i & \text{otherwise}, 
\end{array}\right. 
\]
where $\psi$ is a measurable selection of $\partial h(\cdot, x_\star)$, 
and $\theta_i$ is an element of $\partial \iota_{\mC_i}(x_\star)$. By what
precedes, it is immediate that $\int \|\phi\|^2 d\mu < \infty$. By assumption, there
exists a measurable selection $\varphi$ of $\partial g(\cdot, x_\star)$ such 
that $\int (\varphi(s) + \nabla f(u, x_\star)) \mu(ds) = 0$. Using the first
assumption, we get that the couple $(\varphi(s), \nabla f (u, x_\star))$ is a 
${\mathcal L}^2$ representation of $x_\star$.

\subsection{Proof of Prop.~\ref{affine}}

The assertions about $Z(\mA)$ are straightforward. 
A small calculation shows that  
\begin{gather*}
J_\gamma(s,x) = (I + \gamma H(s))^{-1}(x - \gamma d(s)), \quad \text{and} \\ 
A_\gamma(s,x) = A(s, J_\gamma(s,x)) = 
(I + \gamma H(s))^{-1}(H(s) x + d(s))  . 
\end{gather*} 
Using these expressions, we obtain  
\begin{align*} 
\psi_\gamma(x) &= 
\int \Bigl\{ \ps{A(s, J_\gamma(s,x)) - H(s) x_\star - d(s), 
J_\gamma(s,x) - x_\star} + \gamma \| A(s, J_\gamma(s,x)) \|^2 \Bigr\} \, 
\mu(ds) \\
&= \int \Bigl\{ (J_\gamma(s,x) - x_\star)^T \frac{H(s) + H^T(s)}{2}  
(J_\gamma(s,x) - x_\star) + \gamma \| A(s, J_\gamma(s,x)) \|^2 \Bigr\} \, 
\mu(ds) . 
\end{align*} 

Since $(I + \gamma H(s) )^{-1}$ and $H(s) (I + \gamma H(s) )^{-1}$ are
respectively the resolvent and the Yosida regularization of the linear,
monotone and maximal operator $H(s)$, it holds that 
$\| (I + \gamma H(s) )^{-1} \| \leq 1$, and 
$\| \gamma H(s) (I + \gamma H(s))^{-1} \| \leq 1$.

Denoting as $\|\cdot \|_{S}$ the semi norm associated with any
semidefinite nonnegative matrix $S$, we write 
\begin{align*} 
\psi_\gamma(x) &\geq 
\int \| J_\gamma(s,x) - x_\star \|^2_{(H(s) + H^T(s))/2} \, \mu(ds) \\
&= 
\int 
\left\|(I+\gamma H(s))^{-1} \Bigl( (x - x_\star) 
  - \gamma (H(s) x_\star + d(s) ) \Bigr) \right\|^2_{(H(s) + H^T(s))/2} 
  \, \mu(ds) . 
\end{align*} 
Using the inequality $\| a - b \|^2 \geq 0.5 \| a \|^2 - \| b \|^2$, we obtain
that $\psi_\gamma(x) \geq 0.5 W_\gamma(x) - U_\gamma$, with  
\begin{align*}
W_\gamma(x) &= \int 
\left\|(I+\gamma H(s))^{-1} (x - x_\star) \right\|^2_{(H(s) + H^T(s))/2} 
  \, \mu(ds) ,  \quad \text{and} \\
U_\gamma &= \gamma^2 \int 
\left\| (I+\gamma H(s))^{-1} ( H(s) x_\star + d(s) ) 
  \right\|^2_{(H(s) + H^T(s))/2} \, \mu(ds)  \\
&= \gamma \int \left\|  H(s) x_\star + d(s) ) 
  \right\|^2_{\gamma I_{\gamma}(s)} \, \mu(ds) . 
\end{align*} 

with 
\[
I_\gamma(s) = 
(I+\gamma H(s))^{-T} \frac{H(s) + H^T(s)}{2} (I+\gamma H(s))^{-1} . 
\]

From the inequalities shown above, we have 
\[
\Bigl\| \gamma I_{\gamma}(s) \Bigr\| \leq 1 . 
\]
Therefore, 
\[
0 \leq U_\gamma \leq \gamma \int \| H(s) x_\star + d(s) \|^2 \, \mu(ds) 
\leq \gamma C . 
\] 
Turning to $W_\gamma(x)$, it holds that 
\[
W_\gamma(x) = (x - x_\star)^T \Bigl( \int I_\gamma(s) \, \mu(ds) \Bigr) 
(x - x_\star) , 
\]
Since $\| I_\gamma(s) \| \leq \bigl\| \frac{H(s) + H^T(s)}{2} \bigr\| $ and $I_\gamma(s) \to_{\gamma\to 0} 
(H(s) + H^T(s))/2$, it holds by dominated convergence that 
$\int I_\gamma(s) \, \mu(ds)  \to_{\gamma\to 0} {\bs H} + {\bs H}^T$. If 
${\bs H} + {\bs H}^T > 0$, then there exists $\gamma_0 > 0$ such that 
\[
\inf_{\gamma\in(0,\gamma_0]} \lambda_{\text{min}} 
\left( \int I_\gamma(s) \, \mu(ds) \right) > 0 , 
\]
where $\lambda_{\text{min}}$ is the smallest eigenvalue. 
Thus, Assumption~\ref{psi-increase}--\ref{psi-quad} is verified. 

\subsection{Proof of Prop.~\ref{prop-Dbnd}}

Since $A_\gamma(s,\cdot)$ is $1/\gamma$-Lipschitz, 
$\| A_\gamma(s, x - \gamma B(s,x)) \| \geq \| A_\gamma(s, x)\| - 
\| B(s,x) \| \geq \| A_\gamma(s, x)\| - \| B(s,x) - B(s,x_\star) \| - 
\| B(s,x_\star) \|$. Therefore, 
\begin{align*}
&\psi_\gamma(x) \\
&\geq 
 \int \Bigl\{ 
   \ps{B(s, x) - B(s,x_\star), x - x_\star} 
   - 6\gamma \| B(s,x) - B(s,x_\star) \|^2
 + \gamma \| A_\gamma(s, x - \gamma B(s,x)) \|^2 \Bigr\} \, \mu(ds) \\ 
&\geq 
 \int \Bigl\{ 
   \ps{B(s, x) - B(s,x_\star), x - x_\star} 
   - 8\gamma \| B(s,x) - B(s,x_\star) \|^2
 + (\gamma/2) \| A_\gamma(s, x) \|^2 \\ 
 & 
 \ \ \ \ \ \ \ \ \ \ \ \ \ \ \ \ \ \ \ \ \ \ \ \ \ \ \ \ \ \ \ \ 
 \ \ \ \ \ \ \ \ \ \ \ \ \ \ \ \ \ \ \ \ \ \ \ \ \ \ \ \ \ \ \ \ 
 \ \ \ \ \ \ \ \ \ \ \ \ \ \ \ \ \ 
 - 2\gamma \| B(s,x_\star) \|^2 \Bigr\} \, \mu(ds) \\
&\geq 
 \frac\gamma 2 \int \| A_\gamma(s, x) \|^2 \, \mu(ds) - \gamma C  
\end{align*}
for $\gamma$ small enough, by Hypothesis \ref{cocobnd}. We now have 
\[
 \gamma \int \| A_\gamma(s, x) \|^2 \, \mu(ds) = 
\frac 1\gamma \int \| x - J_\gamma(s, x) \|^2 \, \mu(ds) 
\geq \frac 1\gamma \int d(s,x)^2 \, \mu(ds) 
\geq \frac C\gamma \bs d(x)^2 
\]
thanks to Hypothesis~\ref{linreg-bnd}. The 
result follows from the boundedness of $\mD$. 

\subsection{Proof of Prop.~\ref{coer-gal}}

To prove this proposition, we start with the following result. 

\begin{lemma}
\label{ps->infty} 
Let $\sA \in\maxmon$ be such that 
\[
\exists (x_*,y_*) \in \graph(\sA), \ \exists \delta > 0, \quad  
\bs S(x_*, \delta) \subset \inter(\dom \sA),  \ \text{and} \ 
\forall x\in \bs S(x_*, \delta), \ \inf_{y\in \sA(x)} 
\ps{y - y_*, x - x_*} > 0 .   
\]
Then, assuming that $\dom\sA$ is unbounded, 
\[
\liminf_{x\in\dom \sA, \|x\|\to\infty} 
\frac{\inf_{y \in \sA(x)} \ps{y - y_*, x - x_*}}{\| x \|} > 0 .  
\]
\end{lemma} 
\begin{proof}
Given a vector $u\in  E$, define the function 
\[ 
f_u(\lambda) = \inf_{y\in \sA(x_* + \lambda u)} \ps{y - y_*, u} 
\]
for all $\lambda \geq 0$ such that $x_* + \lambda u \in \dom \sA$.
For all $\lambda_1 > \lambda_2$ in $\dom f_u$, and all 
$y_1 \in A(x_* + \lambda_1 u)$ and $y_2 \in A(x_* + \lambda_2 u)$, we have 
\[
\ps{y_1 - y_*, u} - \ps{y_2 - y_*, u} =  
\ps{y_1 - y_2, u} = \frac{1}{\lambda_1 - \lambda_2} 
\ps{y_1 - y_2, x_* + \lambda_1 u - (x_* + \lambda_2 u) }  
\geq 0 . 
\]
Passing to the infima, we obtain that $f_u(\lambda_1) \geq f_u(\lambda_2)$, in
other words, $f_u$ is non decreasing. 

For all $x \in \dom\sA$ such that $\| x - x_* \| \geq \delta$, we have by
setting $u = \delta (x - x_*) / \| x - x_* \|$
\[
\inf_{y \in \sA(x)} \ps{y - y_*, x - x_*} = 
\frac{\| x - x_* \|}{\delta} f_u(\delta^{-1} \| x - x_* \|) 
\geq \frac{\| x - x_* \|}{\delta} f_u(1) . 
\]
For any $u \in \bs S(0,\delta)$, it holds by assumption that 
$f_u(1) = \inf_{y\in\sA(x_*+u)} \ps{y - y_*, u}$ is positive. We 
shall show that $f_u(1)$ is lower semicontinuous (lsc) as a function of $u$ on
the sphere $\bs S(0,\delta)$.  Since this sphere is compact, $f_u(1)$ attains
its infimum on $\bs S(0,\delta)$, and the lemma will be proven. 

It is well-known that $\sA$ is locally bounded near any point in the interior
if its domain \cite[Prop.~2.9]{bre-livre73} \cite[\S 21.4]{bau-com-livre11}.
Thus, by the closedeness of $\graph(\sA)$, the $\inf$ in the expression of
$f_u(1)$ is attained.  Let $u_n \to u$, and write $f_{u_n}(1) = \ps{y_n - y_*,
u_n}$. By the maximality of $\sA$, we obtain that for any accumulation point $y$
of $(y_n)$ (who exists by the local boundedness), it holds that $(u,y) \in
\graph(\sA)$.  Consequently, $\liminf_n f_{u_n}(1) \geq f_u(1)$, in other
words, $f_u(1)$ is lsc.  
\end{proof}

We now prove Prop.~\ref{coer-gal}. Let us write 
\begin{gather*} 
f(\gamma,s,x) = \frac{\ps{A_\gamma(s, x) - \varphi(s), 
J_\gamma(s,x) - x_\star}}{\| x \|}  
 + \frac{\| x - J_\gamma(s, x) \|^2}{\gamma \| x \|}, \ \text{and} \\ 
g(s, x) = \inf_{\gamma\in (0,1]} f(\gamma,s,x). 
\end{gather*}
Note that $\psi_\gamma(x) / \| x \| = \int f(\gamma,s,x) \, \mu(ds)$. 
We shall show that $\liminf_{\|x\|\to\infty} g(s, x) > 0$ for all $s\in\Sigma$.
Assume the contrary, namely, that there exist $s\in\Sigma$ and 
$\| x_k \| \to\infty$ such that $g(s, x_k) \to 0$. In these conditions, there
exists a sequence $(\gamma_k)$ in $(0,1]$ such that 
$f(\gamma_k, s, x_k) \to 0$. By inspecting the second term in the expression
of $f(\gamma_k,s,x_k)$, we obtain that 
$\| J_{\gamma_k}(s, x_k) \| / \| x_k \| \to 1$. Rewriting the first term as
\[
\frac{\| J_{\gamma_k}(s, x_k) \|}{\| x_k \|}
\frac{\ps{A_{\gamma_k}(s, x_k) - \varphi(s), J_{\gamma_k}(s,x_k) - x_\star}}
{\| J_{\gamma_k}(s, x_k) \|}  , 
\]
and recalling that $A_{\gamma_k}(s, x_k) \in A(s, J_{\gamma_k}(s, x_k))$, 
Lem.~\ref{ps->infty} shows that the $\liminf$ of this term is positive, which
raises a contradiction. 

Note that 
\[
\inf_{\gamma\in (0,1]} \frac{\psi_\gamma(x)}{\| x \|} \geq 
\int g(s,x) \, \mu(ds) . 
\]
Using Fatou's lemma, we obtain 
Assumption~\ref{psi-increase}--\ref{psi-sous-lin}. 


\section{Proofs relative to Section~\ref{sec-apt}} 
\label{anx:apt}

\subsection{Proof of Lem.~\ref{lem:ui}}

Let $\varepsilon$ be the smallest of the three constants (also named $\varepsilon$) in Assumptions~\ref{A0bnd}, \ref{BBnd} and \ref{JBnd-dif}
respectively where $\cK = B_R$.
For every $a,\gamma$, the following holds for $\bar \bP^{a,\gamma}$-almost all $x=(x_n:n\in \bN)$:
\begin{align*}
\bs d(x_{n+1}) \1_{\| x_{n+1} \| \leq R} &= 
\bs d(x_{n+1}) \1_{\| x_{n+1} \| \leq R} ( \1_{\| x_{{n}} \| \leq R}  +
   \1_{\| x_{{n}} \| > R} ) 
=  \bs d(x_{n+1}) \1_{\| x_{n+1} \| \leq R} \1_{\| x_{{n}} \| \leq R} \\ 
&\leq \bs d(x_{n+1}) \1_{\| x_{{n}} \| \leq R} \\
&= \| x_{{n+1}} - \Pi_{\mD}(x_{{n+1}}) \| \1_{\| x_{{n}} \| \leq R} \\
&\leq \| x_{{n+1}} - \Pi_{\mD}(x_{{n}}) \|  \1_{\| x_{{n}}\| \leq R} \,.
 \end{align*} 
Using the notation $\bar \bE_n^{a,\gamma} = \bar \bE^{a,\gamma}(\,.\,|x_0,\dots,x_n)$, we thus obtain:
\begin{align*}
  \bar \bE_n^{a,\gamma}(\bs d(x_{n+1})^{1+\varepsilon}\1_{\| x_{n+1}
    \| \leq R})&\leq \int \| J_\gamma(s,x_n-\gamma B(s,x_n)) -  \Pi_{\mD}(x_{{n}}) \|^{1+\varepsilon} \1_{\| x_{{n}}\| \leq R}\,d\mu(s)\,.
\end{align*}
By the convexity of $\|\cdot\|^{1+\varepsilon}$, for all $\alpha \in (0,1)$, 
\begin{align*} 
\| x+y\|^{1+\varepsilon} &= \frac{1}{\alpha^{1+\varepsilon}} 
\Bigl\| \alpha x + (1-\alpha) \frac{\alpha}{1-\alpha} y \Bigr\|^{1+\varepsilon}
\leq \alpha^{-\varepsilon} \| x\|^{1+\varepsilon} + 
(1-\alpha)^{-\varepsilon} \| y\|^{1+\varepsilon}\,.
\end{align*} 
Therefore, by setting $\delta_\gamma(s,a)\eqdef \|J_\gamma(s,a-\gamma B(s,a)) -  \Pi_{D(s)}(a)\|$, 
\begin{multline*}
  \bar \bE_n^{a,\gamma}(\bs d(x_{n+1})^{1+\varepsilon}\1_{\| x_{n+1}
    \| \leq R}) \leq \alpha^{-\varepsilon} \int \delta_\gamma(s,x_n)^{1+\varepsilon} \1_{\| x_{{n}}\| \leq R}\,d\mu(s) \\ +
(1-\alpha)^{-\varepsilon} \int \| \Pi_{D(s)}(x_{{n}}) -  \Pi_{\mD}(x_{{n}}) \|^{1+\varepsilon} \1_{\| x_{{n}}\| \leq R}\,d\mu(s)\,.
\end{multline*}
Note that for every $s\in \Xi$, $a\in E$,
$$
\|\delta_\gamma(s,a)\| \leq \|J_\gamma(s,a) - \Pi_{D(s)}(a)\| 
 + \gamma \|B(s,a)\|\,.
$$
Hence, by Assumptions~\ref{JBnd-dif} and~\ref{BBnd}, there exists a deterministic constant $C>0$ s.t.
$$
\sup_n \int \delta_\gamma(s,x_n)^{1+\varepsilon} \1_{\| x_{{n}}\| \leq R}\,d\mu(s) \leq C \gamma^{1+\varepsilon}\,.
$$
Moreover, 
since $\Pi_{\cl(D(s))}$ is a firmly non expansive operator 
\cite[Chap.~4]{bau-com-livre11}, it holds that for all  $u \in \cl(\mD)$, and 
for $\mu$-almost all $s$,  
\begin{align*}
\| \Pi_{\cl(D(s))}(x_n) - u \|^2 &\leq \| x_n - u \|^2 - \| \Pi_{\cl(D(s))}(x_n) - x_n \|^2 . 
\end{align*}
Taking $u= \Pi_{\cl(\mD)}(x_n)$, we obtain that 
\begin{equation} 
\label{pi-fne}
\| \Pi_{\cl(D(s))}(x_n) - \Pi_{\cl(\mD)}(x_n) \|^2 \leq 
\bs d(x_n)^2 - d(x_n, D(s))^2 .
\end{equation} 
Making use of Assumption \ref{linreg}, and assuming without loss of generality 
that $\varepsilon \leq 1$, we obtain 
\begin{align*}
   \int \| \Pi_{\cl(D(s))}(x_{{n}}) -  \Pi_{\cl(\mD)}(x_{{n}}) \|^{1+\varepsilon} 
  \,d\mu(s)
&\leq  \left(\int \| \Pi_{\cl(D(s))}(x_{{n}}) -  \Pi_{\cl(\mD)}(x_{{n}}) \|^2 
  \,d\mu(s)\right)^{(1+\varepsilon)/2} \\
&\leq  \alpha'  \bs d(x_n)^{1+\varepsilon}\,,
\end{align*}
for some $\alpha' \in [0, 1)$. Choosing $\alpha$ close enough to zero, we
obtain that there exists $\rho \in [0,1)$ such that 
\begin{align*} 
 \bar \bE_n^{a,\gamma}\left(\frac{\bs d(x_{n+1})^{1+\varepsilon}}{\gamma^{1+\varepsilon}}\1_{\| x_{n+1}
    \| \leq R}\right)
&\leq 
\rho \frac{\bs d(x_{n})^{1+\varepsilon}}{\gamma^{1+\varepsilon}} \1_{\| x_{n} \| \leq R} + C . 
\end{align*} 
Taking the expectation at both sides, iterating, and using the fact that 
$\bs d(x_0) = \bs d(a) < M \gamma$, we obtain that
\begin{equation}
\sup_{{n\in\bN, a\in \cK\cap \mD_{\gamma M} , \gamma\in (0,\gamma_0]}} 
\bar \bE^{a,\gamma}\left(\left(\frac{\bs d(x_{n})}{\gamma}\right)^{1+\varepsilon}\1_{\|x_{n}\|\leq R}\right)< +\infty\,.
\label{eq:d-vanish}
\end{equation}
Since $A_\gamma(s,\cdot)$ is $\gamma^{-1}$-Lipschitz continuous, 
$\| A_\gamma(s, x - \gamma B(s,x)) \| 
\leq  \| A_\gamma(s, x) \| + \| B(s,x) \|$. 
Moreover, choosing measurably $\tilde x \in \mD$ in such a way that 
$\| x - \tilde x \| \leq 2 \bs d(x)$, we obtain 
$\| A_\gamma(s, x) \| \leq \| A_0(s, \tilde x) \| + 2 \frac{\bs d(x)}{\gamma}$.
Therefore, there exists $R'$ depending only on $R$ and $\mD$ s.t.
\[
\| A_\gamma(s, x) \| \1_{\| x\| \leq R} \leq 
\| A_0(s, \tilde x) \| \1_{\| \tilde x\| \leq R'} + 
2 \frac{\bs d(x)}{\gamma} \1_{\| x\| \leq R} \,.
\]
Thus, 
\begin{align}
  \bar\bE^{a,\gamma}_n(\|Z_{n+1}^\gamma\|^{1+\varepsilon}) &= \int \|h_{\gamma,R}(s,x_n)\|^{1+\varepsilon}d\mu(s) \nonumber \\
&= \int \|B(s,x_n) + A_\gamma(s,x_n-\gamma B(s,x_n))\|^{1+\varepsilon}\1_{\| x_n\| \leq R}\,d\mu(s) \nonumber  \\
&\leq \int \left(2\|B(s,x_n)\| + \|A_0(s,\tilde x_n)\| + 2\frac{\bs d(x_n)}\gamma\right)^{1+\varepsilon}\1_{\| x_n\| \leq R'}\,d\mu(s)\,. \label{eq:EspCondZn}
\end{align}
By Assumption~\ref{BBnd}, $\int \|B(s,x_n)\|^{1+\varepsilon}\1_{\| x_n\| \leq R}\,d\mu(s) \leq C$
where the constant $C$ depends only on $\varepsilon$ and $R$.
By Assumption~\ref{A0bnd}, we also have $\int \|A_0(s,x_n)\|^{1+\varepsilon}\1_{\| x_n\| \leq R}\,d\mu(s) \leq C$
for some (other) constant $C$. The third term is controlled by Eq.~(\ref{eq:d-vanish}).
Taking expectations, the bound~\eqref{eq:ui} is established. 

\subsection{Proof of Lem.~\ref{lem:sko}}

The first point can be obtained by straightforward application of Prokhorov and Skorokhod's theorems.
However, to verify the second point, we need to construct the sequences more carefully. 
Choose $\varepsilon>0$ as in Lem.~\ref{lem:ui}.
We define the process $Y^\gamma:E^\bN\to \bR^\bN$ s.t. for every $n\in \bN$,
$$
Y_{n}^\gamma(x) \eqdef \sum_{k=0}^{n-1} \frac{\bs d(x_k)^{1+\varepsilon/2}}{\gamma^{\varepsilon/2}}\1_{\|x_k\|\leq R}\,,
$$
and we denote by $(X,Y^\gamma):E^\bN\to (E\times\bR)^\bN$ the process given by $(X,Y^\gamma)_n(x)\eqdef (x_n,Y^\gamma_{n}(x))$.
We define for every $n$, $\tilde Z_{n+1}^\gamma \eqdef \gamma^{-1}((X,Y^\gamma)_{n+1}-(X,Y^\gamma)_{n})$.
By Lem.~\ref{lem:ui}, it is easily seen that 
$$
\sup_{n\in\bN, a\in \cK\cap \mD_{\gamma M} , \gamma\in (0,\gamma_0]}\bar \bE^{a,\gamma}\left( \|\tilde Z_n^\gamma\|\1_{\| \tilde Z_n^\gamma\|>A}\right) \xrightarrow[]{A\to +\infty}  0\,.
$$
We now apply \cite[Lemma 4.2]{bia-hac-sal-(arxiv)16}, only replacing $E$ by $E\times \bR$ and
$\bar \bP^{a,\gamma}$ by $\bar \bP^{a,\gamma}(X,Y^\gamma)^{-1}$. By this lemma, the family 
$\{\bar \bP^{a,\gamma}(X,Y^\gamma)^{-1}\overline \sX_\gamma^{-1}:
 a\in \cK\cap \mD_{\gamma M} , \gamma\in (0,\gamma_0]\}$ is tight,
where $\overline \sX_\gamma^{-1}:(E\times \bR)^\bN\to C(\bR_+,E\times \bR)$ is the piecewise linear interpolated process,
defined in the same way as $\sX_\gamma$ only substituting $E\times \bR$ with $E$ in the definition.
By Prokhorov's theorem, one can choose the subsequence $(a_n,\gamma_n)$ s.t.
$\bar \bP^{a_n,\gamma_n}(X,Y^{\gamma_n})^{-1}\overline \sX_{\gamma_n}^{-1}$ converges narrowly to some probability measure $\Upsilon$ on $E\times\bR$.
By Skorokhod's theorem, we can define a stochastic process $((\sx_n,\sy_n):n\in \bN)$ on some probability space $(\Omega',\mcF',\bP')$
into $C(\bR_+,E\times \bR)$, whose distribution for a fixed $n$ coincides with $\bar \bP^{a_n,\gamma_n}(X,Y^{\gamma_n})^{-1}\overline \sX_{\gamma_n}^{-1}$, and
 s.t. for every $\omega\in \Omega'$, $(\sx_n(\omega),\sy_n(\omega))\to (\sz(\omega),\sw(\omega))$,
where $(\sz,\sw)$ is a r.v. defined on the same space. In particular, the first marginal distribution of
$\bar \bP^{a_n,\gamma_n}(X,Y^{\gamma_n})^{-1}\overline \sX_{\gamma_n}^{-1}$ coincides with $\bar \bP^{a_n,\gamma_n}\sX_{\gamma_n}^{-1}$. Thus, the first point 
is proven. 
 
For every $\gamma\in (0,\gamma_0]$, introduce the mapping
\begin{eqnarray*}
  \Gamma_\gamma : C(\bR_+,E) &\to& C(\bR_+,\bR) \\
\sx &\mapsto& \left( t\mapsto \int_0^t (\gamma^{-1}\bs d(\sx(\gamma\lfloor u/\gamma\rfloor)))^{1+\varepsilon/2}\1_{\|\sx(\gamma\lfloor u/\gamma\rfloor)\|\leq R}du\right)\ .
\end{eqnarray*}
We denote by  $\underline \sX_\gamma^{-1}: \bR^\bN\to C(\bR_+,\bR)$ the piecewise linear interpolated process,
defined in the same way as $\sX_\gamma$ only substituting $\bR$ with $E$ in the definition. It is straightforward to show that
$\underline \sX_\gamma\circ Y^{\gamma_n} = \Gamma_\gamma\circ \sX_\gamma$.
For every $n$, by definition of the couple $(\sx_n,\sy_n)$, the distribution under $\bP'$ of 
the r.v. $\Gamma_{\gamma_n}(\sx_n)-\sy_n$ is equal to the distribution
of $\Gamma_{\gamma_n}\circ \sX_{\gamma_n} - \underline \sX_{\gamma_n}\circ Y^{\gamma_n}$ under $\bar \bP^{a_n,\gamma_n}$.
Therefore, $\bP'$-a.e. and for every $n$, $\sy_n=\Gamma_{\gamma_n}(\sx_n)$.
This implies that, $\bP'$-a.e., $\Gamma_{\gamma_n}(\sx_n)$ converges (uniformly on compact set)
to $\sw$. On that event, this implies that for every $T\geq 0$, $\Gamma_{\gamma_n}(\sx_n)(T)\to \sw(T)$, which is finite.
Hence, $\sup_n  \Gamma_{\gamma_n}(\sx_n)(T)<\infty$ on that event, which proves the second point.

\subsection{Proof of Lem.~\ref{lem:z-in-D}}

  For every $t\geq 0$, notice that $u_n(t)\to\sz(t)$ $\bP'$-a.e.
Thus, $\bs d(\sz(t))\1_{\|\sz(t)\| < R} \leq \liminf_n \bs d(u_n(t))\1_{\|u_n(t)\|\leq R} \,.$
By Fatou's lemma, $$\bE'(\bs d(\sz(t))\1_{\|\sz(t)\|<R}) \leq \liminf_n \bE'( \bs d(u_n(t))\1_{\|u_n(t)\|\leq R})\,.$$
Define $k_n = \lfloor \frac t{\gamma_n}\rfloor$ and notice that
  \begin{align*}
    \bE'( \bs d(u_n(t))\1_{\|u_n(t)\|\leq R}) &= \bar \bE^{a_n,\gamma_n}(\bs d(x_{k_n})\1_{\|x_{k_n}\|\leq R}) \\
&\leq \sup_{k\in \bN} \bar \bE^{a_n,\gamma_n}(\bs d(x_{k})\1_{\|x_{k}\|\leq R}) \,.
  \end{align*}
By Lem.~\ref{lem:ui} and since $\gamma_n\to 0$, the supremum in the above inequality converges to zero as $n\to\infty$.
As a consequence, $\bE'(\bs d(\sz(t))\1_{\|\sz(t)\|<R})=0$. This means that, $\bP'$-a.e., $\sz(t)\in \bmD\cup B_R^c$.
As $\bmD\cup B_R^c$ is closed and $\sz$ is continuous, the probability-one event on which the above inclusion holds can be made independent from $t$,
and the conclusion follows.

\subsection{Proof of Lem.~\ref{lem:cv-HR}}

Consider any $t\geq 0$ and any $s$ s.t. $\mD\subset D(s)$. We prove that $(u_n(t),v_n(s,t))\to \graph(H_R(s,\,.\,))$.
It is clear that $u_n(t)\to \sz(t)$. If $\|\sz(t)\|\geq R$, the result is trivial. We now assume that $\|\sz(t)\|<R$.
In this case, note that $\sz(t)\in \bmD$ by Lem.~\ref{lem:z-in-D}. This also implies that $\sz(t)\in \cl(D(s))$.

To simplify notations, we now omit the dependence in $(s,t)$ and write $u_n\eqdef u_n(t)$, $v_n\eqdef v_n(s,t)$,
$A\eqdef A(s,\,.\,)$, $B\eqdef B(s,\,.\,)$, $\gamma\eqdef\gamma_n$, $J_\gamma\eqdef J_{\gamma}(s,\,.\,)$, $A_\gamma\eqdef A_{\gamma}(s,\,.\,)$, 
$D\eqdef D(s)$, $H_R=H_R(s,\,.\,)$, $\sz\eqdef\sz(t)$.
We also define $\tilde u_n\eqdef J_\gamma(u_n-\gamma B(u_n))$.

As $\|\sz\|<R$, it holds that $\|u_n\|<R$ for every $n$ large enough. Thus, $-v_n = B(u_n)+A_\gamma(u_n-\gamma B(u_n))$.
We decompose: 
$$
(u_n,-v_n) = (\tilde u_n, B(\tilde u_n)+A_\gamma(u_n-\gamma B(u_n))) + (u_n-\tilde u_n,B(u_n)-B(\tilde u_n))\,.
$$
As $A_\gamma(u_n-\gamma B(u_n))\in A(\tilde u_n)$, the first term in the right hand side belongs to $\graph(A+B)$.
It remains to show that the second term converges to zero, and we deduce that $(u_n,v_n)\to \graph(H_R)$, as obviously
$\graph(-A-B)\subset \graph(H_R)$. One has
\begin{align*}
  \|\tilde u_n-u_n\| &\leq \|J_\gamma(u_n-\gamma B(u_n)) - (u_n-\gamma B(u_n))\| + \gamma \|B(u_n)\| \\
&= \gamma \|A_\gamma(u_n-\gamma B(u_n))\| + \gamma \|B(u_n)\| \\
&\leq \gamma \|A_\gamma(\sz)\|+ \gamma \|A_\gamma(u_n-\gamma B(u_n))-A_\gamma(\sz)\| + \gamma \|B(u_n)\| \\
&\leq \|J_\gamma(\sz)-\sz\|+ \|u_n-\gamma B(u_n)-\sz\| + \gamma \|B(u_n)\| \ ,
\end{align*}
where, for the last inequality, we used the $\gamma^{-1}$-Lipschitz continuity of $A_\gamma$.
As $\sz\in \cl(D(s))$, it holds that $\|J_\gamma(\sz)-\sz\|\to 0$.
Using the continuity of $B$ and the convergence $u_n\to \sz$,  we conclude 
that $\|\tilde u_n-u_n\|\to 0$ as $n\to\infty$. Thus, 
$(u_n-\tilde u_n,B(u_n)-B(\tilde u_n))\to 0$ and the lemma is shown.

\subsection{Proof of Lem.~\ref{lem:v_bounded}}

Define $c_a\eqdef \sup_{a\in B_R\cap \mD}\int \|A_0(s,a)\|^{1+\varepsilon/2}d\mu(s)$
and $c_b\eqdef \sup_{a:\|a\|\leq R}\int \|B(s,a)\|^{1+\varepsilon/2}d\mu(s)$ (these constants
being finite by Assumptions~\ref{A0bnd} and~\ref{BBnd}).
By the same derivations as those leading to Eq.~\eqref{eq:EspCondZn}, we obtain
\begin{align*}
   \int \|v_n(s,t)\|^{1+\varepsilon/2}d\mu(s) &\leq C
\left( \frac{\bs d(u_n(t))^{1+\varepsilon/2}}{\gamma^{1+\varepsilon/2}}\1_{\|u_n(t)\|\leq R}+c_a+c_b\right)\,.
\end{align*}
The proof is concluded by applying Lem.~\ref{lem:sko}.

\subsection{Proof of Lemma \ref{lem:cvth}} 

The sequence $((v_n,\|v_n(\,.\,,\,.\,)\|))$ converges weakly to $(v,w)$ in $\cL^{1+\varepsilon/2}_{E\times\bR}$ along some subsequence
(\emph{n.b.}: compactness and sequential compactness are the same notions in the weak topology of $\cL^{1+\varepsilon/2}_{E\times\bR}$).
We still denote by $((v_n,\|v_n(\,.\,,\,.\,)\|))$ this subsequence.
By Mazur's theorem, there exists a function $J : \mathbb N \to \mathbb N$ and a sequence 
of sets of weights $\{ \alpha_{k,n}: n\in\bN, k=n\ldots, J(n) \, : \, 
\alpha_{k,n} \geq 0, \sum_{k=n}^{J(n)} \alpha_{k,n} = 1 \}$ such 
that the sequence of functions 
$$
(\bar v_n,\bar w_n)\,:\,(s,t)\mapsto \sum_{k=n}^{J(n)} \alpha_{k,n} (v_k(s,t),\|v_k(s,t)\|)
$$
converges strongly to $(v,w)$ in that space, as $n\to\infty$.  Taking
a further subsequence (which we still denote by $(\bar v_n,\bar w_n)$)
we obtain the $\mu\otimes\lambda_T$-almost everywhere convergence of
$(\bar v_n,\bar w_n)$ to $(v,w)$.  Consider a negligible set $\mcN\in
\mcB([0,T])\otimes \mcG$ such that for all $(s,t)\notin \mcN$, the
following assertions are true: \emph{i)} $(\bar v_n(s,t),\bar
w_n(s,t))\to (v(s,t),w(s,t))$; \emph{ii)}
$(u_n(t),v_n(s,t))\to_n \mathrm{gr}(H_R(s,\,.\,))$;  \emph{iii)} $w(s,t)$ is finite.  
The point \emph{ii)} is made possible by Lem.~\ref{lem:cv-HR}.
Let $\varepsilon>0$. By
conditions \emph{ii)} and the fact that $u_n(t)\to\sz(t)$, there exists
$n=n_\varepsilon$ s.t. for all $k\geq n$, there exists $(a_k,b_k)\in
\graph(H_R(s,\,.\,))$ satisfying $\|a_k-\sz(t)\|<\varepsilon$ and
$\|b_k-v_k(s,t)\|<\varepsilon$.  If $\|z(t)\|\geq R$, obviously $(\sz(t),v(s,t))\in \graph(H_R(s,\,.\,))$.
We just need to consider the case where $\|z(t)\|< R$, in which case
the condition $(\sz(t),v(s,t))\in \graph(H_R(s,\,.\,))$
is equivalent to:
\begin{equation}
  \label{eq:sv-graph}
  (\sz(t),- v(s,t)) \in \graph(A(s,\,.\,) + B(s,\,.\,))\,.
\end{equation}
To show Eq.~(\ref{eq:sv-graph}), consider an arbitrary $(p,q)\in \graph(A(s,\,.\,) + B(s,\,.\,))$. Decompose:
\begin{equation}
\ps{q+\bar v_n(s,t), p-\sz(t)} = A_n + B_n + C_n \,,\label{eq:psMaxMon}
\end{equation}
where
\begin{align*}
  A_n & =  \sum_{k=n}^{J(n)}\alpha_{k,n}\ps{q+b_k, p-a_k} \\
B_n &= \sum_{k=n}^{J(n)}\alpha_{k,n}\ps{-b_k+v_k(s,t), p-a_k} \\
C_n &= \sum_{k=n}^{J(n)}\alpha_{k,n}\ps{q+v_k(s,t), a_k-\sz(t)}\,.
\end{align*}
The left hand side of \eqref{eq:psMaxMon} converges to $\ps{q+v(s,t), p-\sz(t)}$.
The term $A_n$ is positive by monotonicity of $A(s,\,.\,) + B(s,\,.\,)$. Moreover, 
\begin{align*}
  B_n&\geq - \varepsilon \sum_{k=n}^{J(n)}\alpha_{k,n}\|p-a_k\| 
\geq -\varepsilon ( \|p\| + \sup_{k\geq n}\|a_k\|)\ \geq -\varepsilon (C+\varepsilon)\,,
\end{align*}
where the constant $C \eqdef \|p\|+\sup_n \|u_n(t)\|$ is finite, since
$u_n(t)$ converges. Similarly,
\begin{align*}
  C_n &\geq - \varepsilon \sum_{k=n}^{J(n)}\alpha_{k,n}(\|q\|+\|v_k(s,t)\|)\,,
\end{align*}
and the right hand side converges to $-\varepsilon (\|q\|+w(s,t))$.
Letting $\varepsilon\to 0$, we conclude that $\ps{q+v(s,t), p-\sz(t)}\geq 0$.
As $A(s,\,.\,) + B(s,\,.\,)\in \maxmon$, this implies that 
Eq.~\eqref{eq:sv-graph} holds.


\begin{thebibliography}{10}

\bibitem{atc-for-mou-14}
Y.~F. {Atchade}, G.~{Fort}, and E.~{Moulines}.
\newblock {On stochastic proximal gradient algorithms}.
\newblock {\em ArXiv e-prints, 1402.2365}, February 2014.

\bibitem{att-79}
H.~Attouch.
\newblock Familles d'op\'erateurs maximaux monotones et mesurabilit\'e.
\newblock {\em Annali di Matematica Pura ed Applicata}, 120(1):35--111, 1979.

\bibitem{aub-cel-(livre)84}
J.-P. Aubin and A.~Cellina.
\newblock {\em Differential inclusions}, volume 264 of {\em Grundlehren der
  Mathematischen Wissenschaften [Fundamental Principles of Mathematical
  Sciences]}.
\newblock Springer-Verlag, Berlin, 1984.
\newblock Set-valued maps and viability theory.

\bibitem{bauschke1997legendre}
H.~H. Bauschke and J.~M. Borwein.
\newblock Legendre functions and the method of random bregman projections.
\newblock {\em Journal of Convex Analysis}, 4(1):27--67, 1997.

\bibitem{bauschke1999strong}
H.~H. Bauschke, J.~M. Borwein, and W.~Li.
\newblock Strong conical hull intersection property, bounded linear regularity,
  {J}ameson’s property ({G}), and error bounds in convex optimization.
\newblock {\em Mathematical Programming}, 86(1):135--160, 1999.

\bibitem{bau-com-livre11}
H.~H. Bauschke and P.~L. Combettes.
\newblock {\em Convex analysis and monotone operator theory in {H}ilbert
  spaces}.
\newblock CMS Books in Mathematics/Ouvrages de Math\'ematiques de la SMC.
  Springer, New York, 2011.

\bibitem{ben-(cours)99}
M.~Bena{\"{\i}}m.
\newblock Dynamics of stochastic approximation algorithms.
\newblock In {\em S\'eminaire de {P}robabilit\'es, {XXXIII}}, volume 1709 of
  {\em Lecture Notes in Math.}, pages 1--68. Springer, Berlin, 1999.

\bibitem{ben-hir-aap99}
M.~Bena{\"{\i}}m and M.~W. Hirsch.
\newblock Stochastic approximation algorithms with constant step size whose
  average is cooperative.
\newblock {\em Ann. Appl. Probab.}, 9(1):216--241, 1999.

\bibitem{ben-hof-sor-05}
M.~Bena{\"{\i}}m, J.~Hofbauer, and S.~Sorin.
\newblock Stochastic approximations and differential inclusions.
\newblock {\em SIAM J. Control Optim.}, 44(1):328--348 (electronic), 2005.

\bibitem{ber-11}
D.~P. Bertsekas.
\newblock Incremental proximal methods for large scale convex optimization.
\newblock {\em Math. Program.}, 129(2, Ser. B):163--195, 2011.

\bibitem{bia-16}
P.~Bianchi.
\newblock Ergodic convergence of a stochastic proximal point algorithm.
\newblock {\em SIAM J. Optim.}, 26(4):2235--2260, 2016.

\bibitem{bia-hac-16}
P.~Bianchi and W.~Hachem.
\newblock Dynamical behavior of a stochastic {F}orward-{B}ackward algorithm
  using random monotone operators.
\newblock {\em J. Optim. Theory Appl.}, 171(1):90--120, 2016.

\bibitem{bia-hac-sal-(arxiv)16}
P.~Bianchi, W.~Hachem, and A.~Salim.
\newblock Constant step stochastic approximations involving differential
  inclusions: Stability, long-run convergence and applications.
\newblock {\em arXiv preprint arXiv:1612.03831}, 2016.

\bibitem{bor-lew-livre06}
J.~M. Borwein and A.~S. Lewis.
\newblock {\em Convex analysis and nonlinear optimization}.
\newblock CMS Books in Mathematics/Ouvrages de Math\'ematiques de la SMC, 3.
  Springer, New York, second edition, 2006.
\newblock Theory and examples.

\bibitem{bre-livre73}
H.~Br\'ezis.
\newblock {\em {Op\'erateurs maximaux monotones et semi-groupes de contractions
  dans les espaces de Hilbert}}.
\newblock North-Holland mathematics studies. Elsevier Science, Burlington, MA,
  1973.

\bibitem{bru-75}
R.~E. Bruck, Jr.
\newblock Asymptotic convergence of nonlinear contraction semigroups in
  {H}ilbert space.
\newblock {\em J. Funct. Anal.}, 18:15--26, 1975.

\bibitem{cas-val77}
C.~Castaing and M.~Valadier.
\newblock {\em Convex analysis and measurable multifunctions}.
\newblock Lecture Notes in Mathematics, Vol. 580. Springer-Verlag, Berlin-New
  York, 1977.

\bibitem{com-pes-siam15}
P.~L. Combettes and J.-C. Pesquet.
\newblock Stochastic quasi-{F}ej\'er block-coordinate fixed point iterations
  with random sweeping.
\newblock {\em SIAM J. Optim.}, 25(2):1221--1248, 2015.

\bibitem{com-pes-pafa16}
P.~L. Combettes and J.-C. Pesquet.
\newblock Stochastic approximations and perturbations in forward-backward
  splitting for monotone operators.
\newblock {\em Pure Appl. Funct. Anal.}, 1(1):13--37, 2016.

\bibitem{devries93}
J.~de~Vries.
\newblock {\em Elements of topological dynamics}, volume 257 of {\em
  Mathematics and its Applications}.
\newblock Kluwer Academic Publishers Group, Dordrecht, 1993.

\bibitem{dieuleveut2017bridging}
A.~Dieuleveut, A.~Durmus, and F.~Bach.
\newblock Bridging the gap between constant step size stochastic gradient
  descent and markov chains.
\newblock {\em arXiv preprint arXiv:1707.06386}, 2017.

\bibitem{fau-rot-10}
M.~Faure and G.~Roth.
\newblock Stochastic approximations of set-valued dynamical systems:
  convergence with positive probability to an attractor.
\newblock {\em Math. Oper. Res.}, 35(3):624--640, 2010.

\bibitem{for-pag-99}
J.-C. Fort and G.~Pag{\`e}s.
\newblock Asymptotic behavior of a {M}arkovian stochastic algorithm with
  constant step.
\newblock {\em SIAM J. Control Optim.}, 37(5):1456--1482 (electronic), 1999.

\bibitem{has-63}
R.Z. Ha'sminskii.
\newblock The average principle for parabolic and elliptic differential
  equations and {M}arkov processes with small diffusions.
\newblock {\em Theor. Probab. Appl.}, 8:1--21, 1963.

\bibitem{hia-um-77}
F.~Hiai and H.~Umegaki.
\newblock Integrals, conditional expectations, and martingales of multivalued
  functions.
\newblock {\em Journal of Multivariate Analysis}, 7(1):149 -- 182, 1977.

\bibitem{hir-76}
J.-B. Hiriart-Urruty.
\newblock {\em About properties of the mean value functional and of the
  continuous infimal convolution in stochastic convex analysis}, pages
  763--789.
\newblock Springer Berlin Heidelberg, Berlin, Heidelberg, 1976.

\bibitem{hu-these-77}
J.-B. Hiriart-Urruty.
\newblock {\em Contributions \`a la programmation math\'ematique: cas
  d\'eterministe et stochastique}.
\newblock Universit\'e de Clermont-Ferrand II, Clermont-Ferrand, 1977.
\newblock Th{\`e}se pr{\'e}sent{\'e}e {\`a} l'Universit{\'e} de
  Clermont-Ferrand II pour obtenir le grade de Docteur {\`e}s Sciences
  Math{\'e}matiques, S{\'e}rie E, No. 247.

\bibitem{kus-yin-(livre)03}
H.~J. Kushner and G.~G. Yin.
\newblock {\em Stochastic approximation and recursive algorithms and
  applications}, volume~35 of {\em Applications of Mathematics (New York)}.
\newblock Springer-Verlag, New York, second edition, 2003.
\newblock Stochastic Modelling and Applied Probability.

\bibitem{lau-(livre)72}
P.-J. Laurent.
\newblock {\em Approximation et optimisation}.
\newblock Hermann, Paris, 1972.
\newblock Collection Enseignement des Sciences, No. 13.

\bibitem{molchanov2006theory}
I.~Molchanov.
\newblock {\em Theory of random sets}.
\newblock Probability and its Applications (New York). Springer-Verlag London,
  Ltd., London, 2005.

\bibitem{neveu1965bases}
J.~Neveu.
\newblock {\em Bases math\'ematiques du calcul des probabilit\'es}.
\newblock Masson et Cie, \'Editeurs, Paris, 1964.

\bibitem{pey-sor-10}
J.~Peypouquet and S.~Sorin.
\newblock Evolution equations for maximal monotone operators: asymptotic
  analysis in continuous and discrete time.
\newblock {\em J. Convex Anal.}, 17(3-4):1113--1163, 2010.

\bibitem{phe-97}
R.~R. Phelps.
\newblock Lectures on maximal monotone operators.
\newblock {\em Extracta Math.}, 12(3):193--230, 1997.

\bibitem{roc-69(mes)}
R.~T. Rockafellar.
\newblock Measurable dependence of convex sets and functions on parameters.
\newblock {\em J. Math. Anal. Appl.}, 28:4--25, 1969.

\bibitem{Roc70}
R.~T. Rockafellar.
\newblock {\em Convex analysis}.
\newblock Princeton Mathematical Series, No. 28. Princeton University Press,
  Princeton, N.J., 1970.

\bibitem{roc-wet-82}
R.~T. Rockafellar and R.~J.-B. Wets.
\newblock On the interchange of subdifferentiation and conditional expectations
  for convex functionals.
\newblock {\em Stochastics}, 7(3):173--182, 1982.

\bibitem{rosasco2014convergence}
L.~Rosasco, S.~Villa, and B.~C. V{\~u}.
\newblock Convergence of stochastic proximal gradient algorithm.
\newblock {\em arXiv preprint arXiv:1403.5074}, 2014.

\bibitem{ros-vil-vu-16}
L.~Rosasco, S.~Villa, and B.~C. V{\~u}.
\newblock A stochastic inertial forward--backward splitting algorithm for
  multivariate monotone inclusions.
\newblock {\em Optimization}, 65(6):1293--1314, 2016.

\bibitem{rot-san-siam13}
G.~Roth and W.~H. Sandholm.
\newblock Stochastic approximations with constant step size and differential
  inclusions.
\newblock {\em SIAM J. Control Optim.}, 51(1):525--555, 2013.

\bibitem{wal-wet-69}
D.~W. Walkup and R.~J.-B. Wets.
\newblock Stochastic programs with recourse. {II}: {O}n the continuity of the
  objective.
\newblock {\em SIAM J. Appl. Math.}, 17:98--103, 1969.

\end{thebibliography}

\def\cprime{$'$} \def\cdprime{$''$} \def\cprime{$'$}

\end{document}